\documentclass[a4paper,11pt]{amsart}
\usepackage{amsfonts,amsmath,amssymb,amsthm,mathrsfs}
\usepackage[latin9]{inputenc}
\usepackage{geometry}

\newcommand{\includepathfigure}{\includegraphics{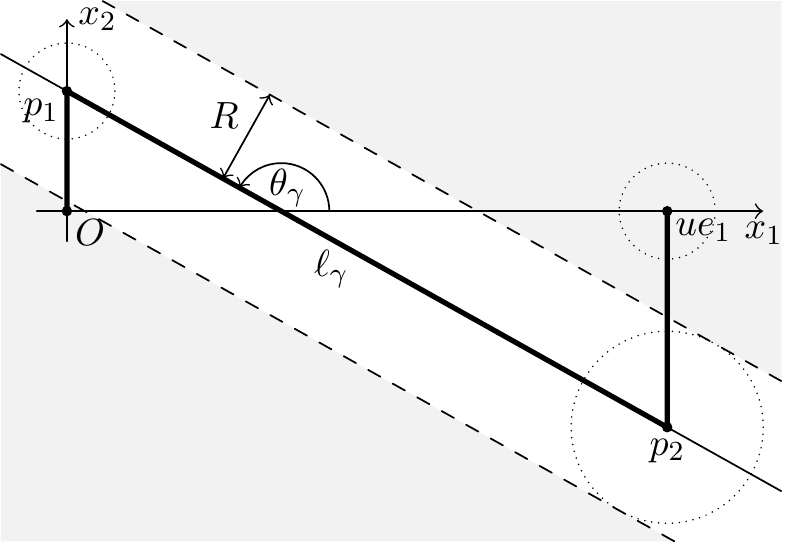}}

\usepackage[colorlinks=true,pagebackref=false]{hyperref}
\usepackage{fixmath}
\usepackage{paralist}
\usepackage{graphicx}
\usepackage{verbatim}
\hypersetup{%
pdftitle={},
pdfsubject={},
pdfauthor={},
pdfkeywords={},
plainpages=false,pdfborder={0 0 0}}

\numberwithin{equation}{section}

\newtheorem{COUNT}{}[section]

\theoremstyle{plain}
\newtheorem{theorem}[COUNT]{Theorem}
\newtheorem{proposition}[COUNT]{Proposition}
\newtheorem{lemma}[COUNT]{Lemma}
\newtheorem*{lemma*}{Lemma}
\newtheorem{corollary}[COUNT]{Corollary}
\theoremstyle{definition}
\newtheorem*{example*}{Example}

\theoremstyle{remark}

\def\skipassumptiona{.3cm}

\def\tabspacing{1.5cm}

\def\assumptionapplications#1#2{\vskip \skipassumptiona
\noindent
\hangindent=\tabspacing
\hangafter=1
\makebox[\tabspacing][l]{#1}#2
\vskip \skipassumptiona}

\def\AAa{\mathcal{A}}
\def\BB{\mathcal {B}}

\def\DD{\mathcal {D}}

\def\FF{\mathcal {F}}
\def\GG{\mathcal {G}}

\def\LL{\mathcal {L}}
\def\MM{\mathcal {M}}

\def\C{\mathbb {C}}

\def\D{\mathbb {D}}
\def\E{\mathbb {E}}
\def\F{\mathbb {F}}
\def\N{\mathbb {N}}
\renewcommand{\P}{\mathbb {P}} 

\def\R{\mathbb {R}}
\def\T{\mathbb {T}}

\def\Z{\mathbb {Z}}

\DeclareMathOperator{\Cov}{Cov}
\DeclareMathOperator{\supp}{supp}

\DeclareMathOperator{\Laplace}{{\Delta}}

\def\leb{\mathscr{L}}

\def\g{\kappa}
\def\modulo{\mathrm{mod\ }}

\def\vmax{{v_{\mathrm{max}}}}
\def\vmin{{v_{\mathrm{min}}}}

\def\FFF{\mathfrak{F}}
\def\DDD{\mathfrak{D}}

\def\PPPhi{\mathfrak{P}}

\def\condT{(T)}

\newcommand{\tor}[1]{{\T,#1}}

\newcommand{\poi}{{\mathrm{poi},\nu}}

\newcommand{\rep}{\textbf C}

\newcommand{\constpot}{c}

\newcommand{\imaginary}{\iota}

\newcommand{\middleconst}{\zeta}

\newcommand{\HHa}{H_1(p_1)}
\newcommand{\HHb}{H^{(2)}}
\newcommand{\HHc}{H^{(3)}}

\newcommand{\exrep}{\hat\rep}
\newcommand{\exOmega}{\hat\Omega}
\newcommand{\exP}{\hat\P}
\newcommand{\extau}{\hat\tau}
\newcommand{\exV}{\hat V}
\newcommand{\exW}{\hat W}
\newcommand{\exE}{\hat \E}
\newcommand{\exFF}{\hat\FF}

\newcommand{\excontV}{\tilde V}
\newcommand{\excontW}{\tilde W}

\makeatletter
\renewcommand{\MR}[1]{}
\makeatother

\begin{document}

\title[Continuity Results and Estimates for the Lyapunov Exponent]{Continuity Results and Estimates for the\\
Lyapunov Exponent of Brownian Motion\\
in Random Potential}

\author{Johannes Rue\ss}
\address{
Mathematisches Institut, Eberhard Karls Universit\"at T\"ubingen\\
Auf der Morgenstelle 10, 72076 T\"ubingen, Germany
}


\begin{abstract}
We collect some applications of the variational formula established in \cite{Schroeder88, Ruess2013} for the quenched Lyapunov exponent of Brownian motion in stationary and ergodic nonnegative potential. We show for example that the Lyapunov exponent for nondeterministic potential is strictly lower than the Lyapunov exponent for the averaged potential. The behaviour of the Lyapunov exponent under independent perturbations of the underlying potential is examined. And with the help of counterexamples we are able to give a detailed picture of the continuity properties of the Lyapunov exponent.
\end{abstract}

\maketitle

\section{Introduction}

In \cite{Schroeder88, Ruess2013} a variational formula has been established for the exponential decay rate of the Green function of Brownian motion evolving in a stationary and ergodic nonnegative potential. The purpose of this article is to collect some applications of this variational formula. A special focus is laid on continuity properties of the Lyapunov exponent. We give counterexamples in the last section in order to complete the picture.

We consider Brownian motion in $\R^d$, $d \in \N$. Let $P_x$ be the law of standard Brownian motion with start in $x \in \R^d$ on the space $\Sigma:=C([0,\infty),\R^d)$ equipped with the $\sigma$-algebra generated by the canonical projections, and let $E_x$ be the associated expectation operator. With $(Z_t)_{t\geq 0}$ we denote the canonical process on $\Sigma$.

We assume that the Brownian motion is moving in a random potential:
Let $(\Omega,\FF,\P)$ be a probability space and assume $(\R^d,+)$ is acting as a group on $\Omega$ via $\tau: \R^d \times \Omega \to \Omega$, $(x,\omega)\mapsto \tau_x\omega$. We always assume that $X:=(\Omega, \FF,\P,\tau)$ is a metric dynamical system, which means that $\tau$ is product measurable and $\P$ is invariant under $\tau_x$ for all $x\in \R^d$. Often $\P$ is required to be ergodic under $\{\tau_x:\, x\in\R^d\}$. Then $X$ is called ergodic dynamical system. We denote the space of $p$-integrable functions on $\Omega$ by $L^p$, $p \geq 1$. Any nonnegative $V\in L^1$ is called \textit{potential} throughout this article.

Let $V$ be a potential. We assume that Brownian motion $Z$ is killed at rate $V$: Introduce the Green function as
\begin{align*}
g(x,y,\omega):= \int_0^\infty p^t(x,y)E_{x,y}^t[\exp\{-\int_0^t V(\tau_{Z_s}\omega) ds\}]dt,
\end{align*}
where $x,y\in \R^d$, $\omega \in \Omega$, $E_{x,y}^t$ denotes the Brownian bridge measure, and $p^t(x,y)$ is the transition probability density of Brownian motion in $\R^d$. $g$ can be interpreted as density for the expected occupation times measure for Brownian motion killed at rate $V$.
Under natural assumptions the Green function is the fundamental solution to
\begin{align*}
-\frac{1}{2}\Laplace g(x,\cdot,\omega) + V_\omega g(x,\cdot,\omega)= \delta_x,
\end{align*}
where $\delta_x$ denotes the Dirac measure at $x \in \R^d$, see e.g.\ \cite[Theorem 4.3.8]{Pinsky1995}.

If $X$ is an ergodic dynamical system and $V$ satisfies certain boundedness and regularity assumptions, then it is shown in \cite[Theorem 1.2]{Ruess2013} that the Green function decays exponentially fast with an deterministic exponential decay rate, called Lyapunov exponent, see also Theorem \ref{app:t:varform}. We refer the reader to \cite{Ruess2013} for the exact assumptions on $V$ needed for this result and we call as in \cite{Ruess2013} a potential satisfying these assumptions shortly a \textit{regular potential}.

Deterministic exponential decay has been shown previously for example for periodic potentials in \cite{Schroeder88}, and for Poissonian potentials in \cite{Sznitman94}. \cite{Armstrong2012} gives analogous results in the context of Hamilton-Jacobi-Bellman equations. In discrete space, \cite{Zerner98} and \cite{Mourrat11} establish existence of Lyapunov exponents for random walks in random potentials.

Measurable functions $f$ on $\Omega$ give rise to functions $f_\omega$ on $\R^d$, called realisations of $f$, defined by $f_\omega(x)=f(\tau_x \omega)$ for $x \in \R^d$ and $\omega \in \Omega$. If $f_\omega$ is differentiable for all $\omega \in \Omega$ we call $f$ (classically) differentiable and we denote the derivative by $(Df)(\omega):=D(f_\omega)(0)$. Let $y \in \R^d$. We recall the variational expression as introduced in \cite[(1.4)]{Ruess2013}, 
\begin{align*}
\Gamma_V(y):=2 \inf_{f \in \F} \left[ \left( \int \frac{|\nabla f|^2}{8f} + Vf d\P \right) \left( \inf_{\phi \in \Phi_y} \int \frac{|\phi|^2}{2f} d\P \right) \right]^{1/2}.
\end{align*}
Here the space $\F$ is the space of probability densities $f \in L^1$ with the following properties:
\begin{compactitem}\label{def:F_s}
\item $\E f = 1$ and there exists $c_f>0$ such that $f \geq c_f$,
\item $f_\omega$ is differentiable of any order for all $\omega$, and $\sup_\Omega |D^n f| < \infty$ for $n \in \N_0$.
\end{compactitem}
The space $\Phi_y$ is the space of divergence-free vector fields $\phi \in (L^1)^d$ such that:
\begin{compactitem}\label{def:P^s}
\item $\phi_\omega$ is differentiable of any order for all $\omega$, and $\sup_\Omega|D^n\phi|<\infty$ for all $n \in \N_0$,
\item $\E \phi = y$, and $\nabla \cdot \phi = 0$ for all $\omega$.
\end{compactitem}

One has the following representation of the Lyapunov exponent:

\begin{theorem}[{\cite[(1.1)]{Schroeder88}, \cite[Theorem 1.2]{Ruess2013}}]\label{app:t:varform}
If $X$ is an ergodic dynamical system and $V$ a \textit{regular potential}, then for all $y\in \R^d \setminus \{0\}$ $\P$-a.s.\ the limit in the following exists and is given as
\begin{align*}
\alpha_V(y):= \lim_{r \to \infty} - \frac{1}{r} \ln g(0,ry,\omega) = \Gamma_V(y).
\end{align*}
\end{theorem} 

In the present article we derive properties of the Lyapunov exponent $\alpha_V$ from its variational representation $\Gamma_V$. We state the results in the following for the variational expression $\Gamma_V$ having in mind that as soon as the underlying dynamical system is ergodic and the considered potentials are \textit{regular} these results do hold by Theorem \ref{app:t:varform} for $\alpha_V$ as well.

\subsection*{Notation}
At some points we use notation from \cite{Ruess2013} and in order to keep this note compact we introduce several objects in a short way and refer the reader to the first two sections of \cite{Ruess2013} for more detailed descriptions.

We denote by $S^{d-1}$ the set of unit vectors in $\R^d$. The Lebesgue measure on $\R^d$ is denoted by $\leb$. We write $|\cdot|$ for the euclidean norm on $\R^k$, $k \in \N$.

We need the concept of weak differentiability on $X$: A measurable function $f:\Omega \to \R^d$ is called weakly differentiable in direction $i$ if $\P$-a.e.\ realisation of $f$ is weakly differentiable in direction $i$, and if there exists a measurable function $g$ on $\Omega$ such that $\P$-a.s.\ $\leb$-a.e.\ $g_\omega=\partial_i(f_\omega)$. Then $g$ is called the weak derivative $\partial_i f$ of $f$ in direction $i$. The weak derivative is uniquely determined $\P$-a.s., and coincides with the classical derivative if the realisations of $f$ are classically differentiable. We have the differential operator $\nabla f = (\partial_i f)_i$, if the weak derivatives in any direction exist. We introduce
\begin{align*}
\DD(\partial_i):=\{f \in L^2:\, f \text{ weakly differentiable in direction $i$}, \, \partial_if \in L^2\}.
\end{align*}
On $\bigcap_i\DD(\partial_i)$ we have the norm $\|f\|_\nabla:= \|f\|_2+\sum_i\|\partial_if\|_2$. In addition to $\F$ and $\Phi_y$ we need the following function spaces: Let $y \in \R^d$, define $\D_w:=\bigcap_{i=1}^d \DD(\partial_i)$, and
\begin{itemize}[]
\item $\D:=\{f \in L^1:\; f_\omega \in C^\infty(\R^d) \; \forall \omega \in \Omega, \; \sup_\Omega |D^nf|<\infty \; \forall n \in \N_0\}$,
\item $\DDD:=\{D \subset \D_w:\; D \text{ dense in } \D_w \; \text{w.r.t.}\; \|\cdot\|_\nabla\}$,
\item $\F_w:=\{f \in \D_w:\; \E f = 1, \; \exists\; c_f>0 \; \text{s.t.}\; f>c_f \;\P\text{-a.s.},\; \|f\|_\infty, \|\nabla f\|_\infty<\infty\}$,
\item $\FFF:=\{F \subset \F_w:\; \forall f \in \F_w \; \exists (f_n)_n \subset F \; \text{and}\; c>0\; \text{s.t.}$\\
\hspace*{2cm} $f_n \to f \;\text{w.r.t.}\; \|\cdot\|_\nabla \text{ and } \inf_n f_n>c \; \P\text{-a.s.}\}$,
\item $\Phi_y^w:=\{\phi \in (L^2)^d: \; \E[\phi \cdot \nabla w]=0 \; \forall w \in \D, \; \E\phi = y\}$,
\item $\PPPhi_y:=\{\phi_y \subset \Phi_y^w: \, \phi_y \text{ dense in } \Phi_y^w \text{ w.r.t. } \|\cdot\|_2\}$.
\end{itemize}

\subsection*{Examples}\label{app:subsec:examples}
We give two main examples for dynamical systems $X=(\Omega,\FF,\P,\tau)$ which fit into our framework. As a special example in Section \ref{sec:counterex} we encounter the Poisson line process.

\textit{$X^{\tor{d}}$ - The $d$-dimensional torus} $\T^d$: Choose $\Omega:=\T^d$, let $\FF:=\BB(\T^d)$ be the Borel $\sigma$-algebra on $\T^d$, and set $\tau_x\omega:= \omega + x (\modulo{} 1)$ for $x \in \R^d$, $\omega \in \T^d$. With $\P$ being the Lebesgue measure $\leb$ the dynamical system $X$ becomes stationary and ergodic.

\textit{Stationary ergodic random measures}: Let $\Omega:=\MM(\R^d)$ be the set of locally finite measures on $(\R^d, \BB(\R^d))$ equipped with the topology of vague convergence. Let $\FF$ be the Borel $\sigma$-algebra on $\Omega$, and set $\tau_x\omega[A]:= \omega[A+x]$ for $x \in \R^d$, $A \in \BB(\R^d)$ and $\omega \in \Omega$. Then for any distribution $\P$ of a stationary ergodic random measure on $(\MM(\R^d),\FF)$ the dynamical system $X$ becomes an ergodic dynamical system, use \cite[Exercise 12.1.1(a)]{Daley2008}. An easy way to construct potentials on $\Omega$ is to choose a `shape' function $W:\R^d \to [0,\infty)$ and set
\begin{align*}
V(\omega):=\int_{\R^d} W(x) \omega(dx).
\end{align*}
If $\P$ is a Poisson point process with constant intensity this leads to the so called Poissonian potentials, see \cite{Sznitman98}. We denote such a dynamical system where $\P$ is a Poisson point process with constant intensity $\nu>0$ by $X^{\poi}$.

\section{Elementary Properties} \label{subsec:basicproperties}

We deduce elementary properties of $\Gamma_V$:

\begin{proposition}\label{prop:application_basic}
Assume $V$ is a potential. For $c\geq 0$, for $y \in \R^d$,
\begin{align}\label{app:e:linearity}
\Gamma_V(cy) = c \Gamma_V(y).
\end{align}
Let $c \geq 1$, then
\begin{align}\label{e:cV}
\Gamma_{cV}^2 \leq c \Gamma_V^2.
\end{align}
Analogously if $0 \leq c \leq 1$, one has $\Gamma_{cV}^2 \geq c \Gamma_V^2$. In constant potential $c\geq 0$, for $y \in \R^d$,
\begin{align}\label{e:lyapunovExponent_forconstpot}
\Gamma_c(y) = \sqrt{2c} |y|.
\end{align}
$\Gamma$ is concave in the following sense: Let $\lambda_i$, $1 \leq i \leq k$, be positive real numbers s.t.\ $\sum_{i=1}^k \lambda_i = 1$. Let $V_1, \ldots V_k$ be potentials on $\Omega$, then
$\Gamma_{\sum_{i=1}^k \lambda_i V_i}^2 \geq \sum_{i=1}^k \lambda_i \Gamma_{V_i}^2$,
in particular,
\begin{align} \label{e:concave}
\Gamma_{\sum_{i=1}^k \lambda_i V_i} \geq \sum_{i=1}^k \lambda_i \Gamma_{V_i}.
\end{align}
For sums of constant potential $c$ and some other potential $V$
\begin{align}\label{e:c+Vgeq}
\Gamma_{c+V}^2 \geq \Gamma_V^2 + \Gamma_c^2.
\end{align}
If $\sigma_s(0):=\inf_{f \in \F} \E[|\nabla f|^2/(8f)+Vf]>0$, we have for $x$, $y \in \R^d$,
\begin{align}\label{e:alpha_triangle}
\Gamma_V(x+y) \leq \Gamma_V(x)+\Gamma_V(y).
\end{align}
$\Gamma_V$ is monotone in $V$: Assume $V_1 \leq V_2$ are potentials, then
\begin{align}\label{eq:Gamma_monotone}
\Gamma_{V_1} \leq \Gamma_{V_2}.
\end{align}
For $y \in \R^d \setminus\{0\}$,
\begin{align}\label{eq:comparison_with_qfe}
\Gamma_V^2(y/|y|) \geq 2 \sigma_s(0).
\end{align}
\end{proposition}

For completeness we restated \eqref{app:e:linearity} which is shown in \cite[Lemma 3.1]{Ruess2013}. Many of these properties are already established for the Lyapunov exponent of Brownian motion in Poissonian potential, see e.g.\ \cite[Chapter 5]{Sznitman98}. In the discrete space setting of random walk in random potential such results are obtained in \cite[Proposition 4]{Zerner98}. Formula \eqref{e:lyapunovExponent_forconstpot} for the Lyapunov exponent of constant potential is well known, a calculation can be found in \cite[(2.9)]{Ruess2012}.

Inequality \eqref{e:cV} is stronger than inequality $\alpha_{cV} \leq c\alpha_V$, $c \geq 1$, which one obtains applying Jensen inequality to the representation of the Lyapunov exponent given in \cite[(1.9)]{Ruess2013}. This here allows to deduce the correct asymptotics given in \eqref{e:convspeed}. In the same way \eqref{e:concave} could be deduced with Hölder inequality from \cite[(1.9)]{Ruess2013}. The `squared' inequality however is a stronger result.

Inequality \eqref{eq:comparison_with_qfe} can be interpreted as a relation between the Lyapunov exponent and the quenched free energy $\Lambda_\omega(0):=\limsup_{t \to \infty} \frac{1}{t}\ln E_0[ \exp\{-\int_0^t V_\omega (Z_s) ds \}]
$: For example in \cite[Corollary 1.4]{Ruess2013} under suitable assumptions we could relate the quenched free energy of Brownian motion with drift $\lambda \in \R^d$ in potential $V$ to the variational expression $\sigma_s$ with an additional `drift term'.

\begin{proof}
For \eqref{e:cV} note that
\begin{align*}
\int \frac{|\nabla f|^2}{8f} +c Vf d\P \leq c \int \frac{|\nabla f|^2}{8f} + Vf d\P.
\end{align*}

For \eqref{e:lyapunovExponent_forconstpot} recall the `inverse' Hölder inequality: If $g$, $h$ are measurable, $h \neq 0$ $\P$-a.s., then for $r \in (1,\infty)$,
\begin{align}\label{eq:holder_inverse}
\E[|g|^{1/r}]^r \E[|h|^{-1/(r-1)}]^{-(r-1)} \leq \E[|gh|].
\end{align}
This follows from an application of Hölder's inequality $\|f_1f_2\|_1 \leq \|f_1\|_p\|f_2\|_q$, $1\leq p,q \leq \infty$, $p^{-1}+q^{-1}=1$ to $f_1:=|gh|^{1/r}$, $f_2:=|h|^{-1/r}$, $p=r$, $q=r/(r-1)$. \eqref{eq:holder_inverse} applied to $r=2$, $g:=|\phi|^2$, $h:=f^{-1}$, and Jensen inequality give
\begin{align}\label{eq:B_lowerbound}
\inf_{\phi \in \Phi_y} \int \frac{|\phi|^2}{2f}d\P \geq |y|^2/2.
\end{align}
Estimate with \eqref{eq:B_lowerbound}
\begin{align*}
\Gamma_{c}^2(y)
& = 4 \inf_{f \in \F} \left( \int \frac{|\nabla f|^2}{8f} + cf d\P \right) \left( \inf_{\phi \in \Phi_y} \int \frac{|\phi|^2}{2f} d\P \right)
\geq 4c \inf_{f \in \F} \inf_{\phi \in \Phi_y} \int \frac{|\phi|^2}{2f} d\P
\geq 2c|y|^2.
\end{align*}
On the other hand, choosing $f \equiv 1$ and $\phi \equiv y$ in the variational expression for $\Gamma_c^2(y)$, one has $\Gamma_c^2(y) \leq 2c|y|^2$.

For the inequality preceding \eqref{e:concave} observe that:
\begin{align*}
\Gamma_{\sum_i \lambda_i V_i}^2(y)
& = 4\inf_{f \in \F} \inf_{\phi \in \Phi_y}  \left( \int \sum_i \lambda_i \left( \frac{|\nabla f|^2}{8f} + V_i f \right) d\P \right) \left( \int \frac{|\phi|^2}{2f} d\P \right)\\
& \geq \sum_i \lambda_i 4\inf_{f \in \F} \inf_{\phi \in \Phi_y} \left( \int \frac{|\nabla f|^2}{8f} +  V_i f d\P \right) \left( \int \frac{|\phi|^2}{2f} d\P \right)
= \sum_i \lambda_i \Gamma_{V_i}^2(y).
\end{align*}
Since the square root is concave and monotone the `non-squared' inequality \eqref{e:concave} is valid.

For \eqref{e:c+Vgeq} use \eqref{eq:B_lowerbound}, \eqref{e:lyapunovExponent_forconstpot} and
\begin{align*}
\Gamma_{c+V}^2(y)
& \geq 4 \inf_{f \in \F}  \left\{ \left( \int \frac{|\nabla f|^2}{8f} + Vf d\P \right) \inf_{\phi \in \Phi_y} \left( \int \frac{|\phi|^2}{2f} d\P \right) \right\} + 4c \inf_{f \in \F} \inf_{\phi \in \Phi_y} \int \frac{|\phi|^2}{2f} d\P\\
& \geq \Gamma_V(y) + 2c|y|^2.
\end{align*}

For \eqref{e:alpha_triangle} we argue with the representation of $\Gamma_V$ given by \cite[Propositions 3.13, 3.15]{Ruess2013}: For $f\in \F$, $\eta \in S^{d-1}$ let $K(f):= \E[|\nabla f|^2/(8f) + Vf]$ and $H(\eta,f):=\inf_{w \in \D}\E[|\nabla w -\eta|^2f]$. Then
\begin{align*}
\Gamma_V(x+y)
& = \sup_{\eta \in S^{d-1}} |\langle x+y,\eta \rangle| \inf_{f \in \F} 
\left[\frac{2K(f)}{H(\eta,f)}\right]^{1/2}\\
& \leq \sup_{\eta \in S^{d-1}} \left(|\langle x,\eta \rangle| + |\langle y,\eta \rangle| \right)\inf_{f \in \F} 
\left[\frac{2K(f)}{H(\eta,f)}\right]^{1/2}
\leq \Gamma_V(x) + \Gamma_V(y). \qedhere
\end{align*}

For \eqref{eq:Gamma_monotone} note that $\E[ |\nabla f|^2/(8f) + Vf] \E[|\phi|^2/(2f)]$ is monotone in $V$ for all $f \in \F$.

For \eqref{eq:comparison_with_qfe} use \eqref{eq:B_lowerbound} and \eqref{app:e:linearity}.
\end{proof}

\section{Inequalities}

\subsection{Effect of Randomness}

In \cite[Corollary 1.3]{Ruess2013} as a direct consequence of Theorem \ref{app:t:varform} we have seen that
\begin{align}\label{eq:GammaV_leq_GammaEV}
\Gamma_V \leq \Gamma_{\E V}.
\end{align}

The following theorem is a refinement of \eqref{eq:GammaV_leq_GammaEV}.

\begin{theorem}\label{th:LE_for_subsigmaalgebra}
Let $X:=(\Omega,\FF,\P,\tau)$ and $Y:=(\Omega,\GG,\P|_\GG,\tau)$ be metric dynamical systems with $\GG \subset \FF$. Let $V$ be a potential on $X$. Then with obvious notation,
\[
\Gamma_V^X \leq \Gamma_{\E[V|\GG]}^Y.
\]
\end{theorem}

\begin{proof}
Let $\F_s^X$, $\Phi_y^X$ and $\F_s^Y$, $\Phi_y^Y$ denote the spaces $\F$, $\Phi_y$ for the dynamics of $X$ and $Y$ respectively. One has 
$
\F_s^Y \subset \F_s^X \text{ and } \Phi_y^Y\subset\Phi_y^X
$. 
Hence,
\begin{align*}
& \inf_{f \in \F_s^X} \E\left[ \frac{|\nabla f|^2}{8f} + Vf \right]\inf_{\phi \in \Phi_y^X} \E\left[\frac{|\phi|^2}{2f}\right]
\leq 
\inf_{f \in \F_s^Y} \E\left[ \frac{|\nabla f|^2}{8f}  + \E[V|\GG]f\right]\inf_{\phi \in \Phi_y^Y} \E\left[\frac{|\phi|^2}{2f}\right]
\end{align*}
which shows the statement.
\end{proof}

\subsection{Strict Inequality}

It is natural to ask whether the randomness of the potential has significant effect on the Lyapunov exponent. The following theorem gives a positive answer to this question:

\begin{theorem}\label{prop:strictinequality}
Assume $V$ is nondeterministic, weakly differentiable with $\|\nabla V \|_\infty < \infty$. Assume $0<\vmin \leq V \leq \vmax < \infty$. Then for $y \neq 0$,
\begin{align*}
\Gamma_V(y) < \Gamma_{\E V}(y).
\end{align*}
\end{theorem}

\begin{proof}
Without restriction we consider the set of functions $\F_w$ instead of $\F$ in the definition of $\Gamma_V$. This is possible by \cite[Proposition 2.2]{Ruess2013}. Let $0<p<\infty$ and choose $f_p:= \beta V^{-p}$ with $\beta := \E[V^{-p}]^{-1}$. Then $f_p\in \F_w$. One has
\begin{align*}
\nabla f_p = - \beta p V^{-p-1} \nabla V,
\end{align*}
use the chain rule for weak derivatives, see \cite[Lemma 7.5]{Gilbarg83}.
Choosing $\phi \equiv y$, we get
\begin{align}\label{eq:estimate0}
\Gamma_V^2(y)
& \leq 2 |y|^2 \E \left[ \frac{|\nabla f_p|^2}{8f_p} + Vf_p \right] \E\left[\frac{1}{f_p}\right] \nonumber\\
& = 2 |y|^2 \E\left[\beta p^2 \frac{|\nabla V |^2 }{8 V^{p+2}} +\beta V^{1-p} \right] \E\left[\frac{V^p}{\beta}\right]
 = 2 |y|^2 \E\left[p^2 \frac{|\nabla V |^2 }{8 V^{p+2}} + V^{1-p} \right]\E[V^p].
\end{align}

We start considering $\psi(p):=\E[V^{1-p}]\E[V^p]$. $\psi$ is differentiable on $\R$ with derivative 
\begin{align*}
\psi'(p) = - \E[V^{1-p} \ln V] \E[V^p] + \E[V^{1-p}]\E[V^p \ln V],
\end{align*}
where we used the theorem on differentiation under the integral sign, see e.g.\ \cite[Lemma 16.2]{Bauer2001}. At $p=0$ one has
\begin{align*}
\psi'(0)= -\E[V \ln V] + \E[V]\E[\ln V ] = - \Cov(V,\ln V).
\end{align*}
One has $\Cov(V,\ln V)>0$ by FKG-inequality, see e.g.\ \cite[Theorem 1.2]{Rinott1992}.
Therefore, choosing a constant $\Cov(V,\ln V) > C_1 > 0$ one has for $p>0$ small enough,
\begin{align}\label{eq:estimate1}
\psi(p)=\E[V^{1-p}]\E[V^p] < \E[V]-C_1p.
\end{align}

On the other hand there is a constant $C_2>0$ such that for $p>0$,
\begin{align}\label{eq:estimate2}
\E\left[p^2 \frac{|\nabla V|^2}{8 V^{p+2}}\right]\E[V^p] \leq 8^{-1} p^2 \vmin^{-2} (\vmax/\vmin)^p \|\nabla V\|_2^2 \leq C_2 p^2.
\end{align}

Estimates \eqref{eq:estimate0}, \eqref{eq:estimate1} and \eqref{eq:estimate2} give for $p>0$ small
\begin{align*}
\Gamma_V^2(y) \leq 2|y|^2(\E[V]-C_1 p + C_2 p^2),
\end{align*}
which for $p>0$ small enough is strictly lower than $\Gamma_{\E V}^2(y)=2\E[V]|y|^2$, see \eqref{e:lyapunovExponent_forconstpot}.
\end{proof}

\subsection{Perturbation and Extension} \label{subsec:independent_input}

The potential $V$ may be perturbed by an external input or extended into `new' dimensions. Extensions are of interest if one considers for example random chessboard potentials, see e.g.\ \cite{DalMaso1986}. In the following we elaborate a framework for external input and extensions and give estimates on Lyapunov exponents for perturbed or extended potentials.

Let $X:= (\Omega_1, \FF_1, \P_1, \tau^{(1)})$ and $Y:=(\Omega_2, \FF_2, \P_2, \tau^{(2)})$ be metric dynamical systems of dimensions $d_1$ and $d_2$ respectively. Consider some measure $\P$ on $(\Omega, \FF):= (\Omega_1 \times \Omega_2, \FF_1 \otimes \FF_2)$ with marginal distributions $\P_1$ and $\P_2$. We introduce two possible actions on $\Omega$:\\
\noindent \textit{\underline{E}xtension of $X$:} Define for $x=(x_1,x_2) \in \R^{d_1} \times \R^{d_2}$ and $\omega \in \Omega$ the action
\begin{align*}
\tau_x^e\omega:= (\tau_{x_1}^{(1)} \omega_1, \tau_{x_2}^{(2)} \omega_2).
\end{align*}
\noindent \textit{\underline{P}erturbation of $X$:} Assume $d_1 = d_2$, define for $x \in \R^{d_1}$ and $\omega \in \Omega$
\begin{align*}
\tau_x^p\omega:= (\tau_x^{(1)} \omega_1, \tau_x^{(2)} \omega_2).
\end{align*}
Note that $\tau^e$ as well as $\tau^p$ are indeed product measurable actions on the respective spaces.

If $\P$ is invariant under $\tau^p$ or $\tau^e$ then $\P$ is called a joining of $\P_1$ and $\P_2$. We denote the set of joinings with respect to $\tau^p$ by $J_p(X,Y)$ and the set of joinings with respect to $\tau^e$ by $J_e(X,Y)$. The product measure $\P_1 \otimes \P_2$ is always a joining with respect to $\tau^p$ and $\tau^e$. 

Note that joinings of ergodic dynamical systems are not necessarily ergodic any more. For example consider the torus $X^{\tor{d}}$, see page \pageref{app:subsec:examples}. On $(\T^d\times\T^d,\BB(\T^d)\otimes\BB(\T^d),\leb\otimes\leb)$ the shift $\tau^p$ is not ergodic.

Joinings are discussed in literature in great extent. Existence and ergodicity of joinings in general, and the question when the product measure leads to an ergodic joining are addressed e.g.\ in \cite[Chapter 5,6]{Furstenberg1981}, \cite[Chapter 10]{Cornfeld1982}, \cite[Chapter 6]{Rudolph1990}, \cite{Ryzhikov1991}. However, we want to mention that often in literature actions of only one transformation or actions of $\Z$ on $\Omega$ are considered, instead of studying the action of more general groups.

We recall a result in this direction: The definition of weak mixing for one shift can be found in \cite[Definition 4.1]{Rudolph1990}.

\begin{lemma}[see e.g.\ {\cite[Proposition 4.19]{Rudolph1990}}]\label{lem:weakmixing-charakterisation}
Let $\phi$ be a measurable transformation of a probability space $(\tilde \Omega, \tilde \FF,\tilde \P)$.
Then $\phi$ is weakly mixing under $\tilde \P$, if and only if the product $\phi \times \psi$ of $\phi$ with any other ergodic transformation $\psi$ of some probability space $(\hat \Omega, \hat \FF, \hat \P)$ is an ergodic transformation on $(\tilde \Omega \times \hat \Omega, \tilde \FF \otimes \hat \FF, \tilde \P \otimes \hat \P)$.
\end{lemma}

Here the product of $\phi$ and $\psi$ is defined by $\phi\times\psi: \tilde\Omega\times\hat\Omega \to\tilde\Omega\times\hat\Omega$, $(\omega_1,\omega_2) \mapsto (\phi(\omega_1),\psi(\omega_2))$. On page \pageref{app:subsec:examples} we have introduced the ergodic dynamical system $X^{\poi}$, where $\P$ is a Poisson point process. We can construct the following example:

\begin{example*}
The perturbation or extension of $X^{\poi}$ with any other ergodic dynamical system is again an ergodic dynamical system under the product measure.

In fact, with \cite[12.3.II]{Daley2008} considering bounded Borel measurable subsets of $\R^d$ we know that the Poisson point process satisfies \cite[12.3.I(iii)]{Daley2008}. In particular, according to \cite[Definition 4.1]{Rudolph1990} any transformation $\tau_x$, $x \neq 0$, is weakly mixing under $\P$. Therefore, with help of Lemma \ref{lem:weakmixing-charakterisation} the statement follows.
\end{example*}

We need some additional notation: With $\E_1$ we denote the expectation operator with respect to $\P_1$. We write $\Gamma^{\P_1}$, $\Gamma^{\P,e}$ and $\Gamma^{\P,p}$ for the variational functional in order to indicate the underlying dynamical system. $\pi_1: \Omega_1\times\Omega_2\to\Omega_1$ denotes the projection onto $\Omega_1$. For $y \in \R^{d_1}$ we set
\[
\hat y := (y_1, \ldots ,y_{d_1}, 0,\ldots 0) \in \R^{d_1+d_2}.
\]
The following result studies the effect of external input:

\begin{theorem}\label{lem:independent_input}
Let $V$ be a potential on $\Omega$. For any joining $\P \in J_e(X,Y)$, for any $y \in \R^{d_1}$,
\begin{align*}
\Gamma_{V}^{\P,e}(\hat y) \leq \Gamma_{\E[ V|\pi_1=\cdot]}^{\P_1}(y).
\end{align*}
If $d_1=d_2$, the analogous inequality is valid for any joining $\P \in J_p(X,Y)$.
\end{theorem}

Note that in fact, $\Gamma^{\P_1}$ does not depend on the realisation of $\E[ V|\pi_1=\cdot]$.

\begin{proof}
We prove the statement for $\tau^e$. The same argument works for $\tau^p$. Introduce $\F_1$ as the set $\F$ for the dynamical system $X$ as defined on page \pageref{def:F_s}. By $\F_e$ we denote the set $\F$ on $\Omega$. Introduce $\Phi_y^1$ as the set $\Phi_y$ for $X$ and $\Phi_{\hat y}^e$ as the set $\Phi_{\hat y}$ on $\Omega$, see page \pageref{def:P^s}. We define
\begin{align*}
\F^X&:=\{f \in \F_e: \; \forall\; \omega \in \Omega_1 \; \exists\; c_\omega >0 \;\text{s.t.}\; f(\omega,\cdot) \equiv c_\omega \},\\
\Phi_{\hat y}^X &:=\{\phi \in \Phi_{\hat y}^e: \; \forall \;\omega \in \Omega_1 \; \exists \;y_\omega \in \R^{d_1} \;\text{s.t.}\; (\phi_i(\omega,\cdot))_{i=1,\ldots d_1} \equiv y_\omega, \;\forall\; d_1<i \leq d_2 : \; \phi_i \equiv 0\}.
\end{align*}
Considering only the first component any $f \in \F^X$ can be identified uniquely with $\tilde f \in \F_1$ such that $f=\tilde f \circ \pi_1$. Then $|\nabla^{\tau_e}f|^2=|\nabla^{\tau^{(1)}}\tilde f|^2 \circ \pi_1$ with obvious notation. Analogously any $\phi \in \Phi_{\hat y}^X$ can be identified uniquely with $\tilde \phi \in \Phi_y^1$ after a projection of $\phi $ onto its first $d_1$ components such that $(\phi_i)_{i=1,\ldots d_1} = (\tilde \phi \circ \pi_1)_{i=1,\ldots d_1}$. Then $|\phi|^2=|\tilde \phi \circ \pi_1|^2$ and we get
\begin{align*}
(\Gamma_{V}^{\P,e})^2(\hat y)
& \leq 4 \inf_{f \in \F^X} \inf_{\phi \in \Phi_{\hat y}^X}
\E\left[ \frac{|\nabla f|^2}{8f} + Vf \right]
\E\left[ \frac{|\phi|^2}{2f}  \right]\\
& = 4 \inf_{\tilde f \in \F_1} \inf_{\tilde \phi \in \Phi_y^1} 
\E_1\left[\frac{|\nabla \tilde f|^2}{8\tilde f} + \E[V|\pi_1=\cdot] \tilde f\right]
\E_1\left[ \frac{|\tilde \phi|^2}{2\tilde f} \right]
= (\Gamma_{\E[V|\pi_1=\cdot]}^{\P_1})^2(y).
\end{align*}
This shows the statement.
\end{proof}

We use this result to study sums and products of independent potentials.

\begin{corollary}\label{cor:applications_of_independent_input}
Let $\P=\P_1\otimes\P_2$. Assume $V_1, V_2 \in L^1(\P)$ with $V_1$ constant in the second component and $V_2$ constant in the first component. Then for $y \in \R^{d_1}$, 
\begin{align*}
\Gamma_{V_1+V_2}^{\P,e}(\hat y) \leq \Gamma_{V_1+ \E V_2}^{\P_1} (y), \ \ \Gamma_{V_1V_2}^{\P,e}(\hat y) \leq \Gamma_{V_1\E V_2}^{\P_1}(y),
\end{align*}
where for the first inequality $V_1+V_2$ and for the second $V_1V_2$ is required to be a potential. Analogous results hold for the action $\tau^p$.
\end{corollary}

\section{Continuity} \label{sec:continuity}

In this section we study continuity properties of the Lyapunov exponent. We consider continuity with respect to the underlying probability measure, continuity with respect to the potential and we are also interested in the exact rate of convergence of the Lyapunov exponent for scaled potentials. In Section \ref{sec:counterex} we give examples which show that the continuity results we obtain here are essentially all one can expect in general. Additional assumptions however should allow to derive stronger results. Possible enforcements of the prerequisites are for example  mixing properties of the underlying probability measure, finite range dependence properties of the potential, or compactness of the space $\Omega$. We show in Subsection \ref{subsec:semicontinuity_compact} that compactness allows to deduce exact results. Both, compactness assumptions as well as additional mixing or independence properties are studied in literature in comparable situations:

For example for the time constant in i.i.d.\ first-passage percolation continuity has been investigated in \cite[Theorem 3]{Cox1981}, see also \cite[Chapter X.4]{Smythe1978}. Recently, continuity of the Lyapunov exponent of random walk in i.i.d.\ random potential with respect to convergence in distribution of the underlying potential has been shown in \cite{Le2013}. Models with long range dependencies are considered e.g.\ in \cite{Scholler2011}. We also want to refer to \cite[Section 11]{Mourrat11} where similar questions are addressed. In \cite[Lemma 3.1]{Rassoul2012} continuity of the quenched free energy of random walk in i.i.d.\ potential with respect to $L^p$ convergence, $p>d$, of the potential is established. Continuity of quantities similar to the Lyapunov exponent is studied e.g.\ in \cite{Bourgain2002,Bourgain2005,Jitomirskaya2011,Duarte2013,You2013}. There, compactness is a central feature in order to obtain continuity properties.

It is immediate to show continuity of the Lyapunov exponent with respect to uniform convergence of the potential:

\begin{proposition}\label{prop:cont_supnorm}
Let $V$ and $V'$ be potentials. Assume $V \geq \vmin >0$ and $\|V'-V\|_\infty < \vmin$. Then for $y \in \R^d$,
\begin{align*}
|\Gamma_V^2(y)-\Gamma_{V'}^2(y)| \leq \|V'-V\|_\infty \Gamma_V^2(y)/\vmin.
\end{align*}
\end{proposition}

\begin{proof}
Let $\epsilon:=\|V'-V\|_\infty$. Then $V' \leq V+\epsilon \leq V(1+\epsilon/\vmin)$. Now use \eqref{e:cV} in order to get the lower bound. The upper bound follows analogously from $V' \geq V-\epsilon \geq V(1-\epsilon/\vmin)$ and the corresponding inequality after \eqref{e:cV}.
\end{proof}

Consideration of continuity with respect to weak convergence of the potential as well as continuity with respect to the underlying measure turn out to be more delicate. While we are able to show upper semi-continuity, see Subsection \ref{subsec:semicontinuity}, lower semi-continuity does not hold in general as indicated by examples given in Section \ref{sec:counterex}. This resembles the situation in \cite{Cox1981,Le2013} where the proof of the lower bound was more involved than the proof of the upper bound.

\subsection{Denseness}

In Subsection \ref{subsec:semicontinuity} we study continuity of $\Gamma_V$ with respect to weak convergence of the underlying probability measure $\P$ on $\Omega$, and we therefore need to introduce function spaces of continuous functions. Assume $\Omega$ is a topological space such that $\FF$ is the Borel $\sigma$-algebra. We set
\begin{compactitem}[]
\item $\D^c :=\{f \in \D: \;\forall n \in \N_0 \; D^n f \;\text{is continuous w.r.t.\ the topology on}\; \Omega\},$
\item $\F^c := \F \cap \D^c$,
\item $\Phi_y^c := \Phi_y\cap (\D^c)^d$.
\end{compactitem}
We need the following condition on $(\Omega,\FF,\P,\tau)$:
\assumptionapplications{\condT{}}
{\label{item:condT}$\Omega$ is a completely regular, first countable Hausdorff space s.t.\ $\FF$ is the Borel $\sigma$-algebra, $\P$ is a Radon measure, the mapping $\omega \mapsto \tau_x\omega$ is continuous for all $x$.}

In \cite[Proposition 2.2]{Ruess2013} it is shown that if $V \in L^2$ we may replace the function spaces in the definition of $\Gamma_V(y)$ by any of the sets in $\FFF$ and $\PPPhi_y$ without changing $\Gamma_V(y)$.


\begin{proposition} \label{lem:spacesdense_continuous}
Assume that $(\Omega,\FF,\P,\tau)$ satisfies \condT{}. Then $\D^c$ is dense in $L^2$.
Moreover, 
$\D^c \in \DDD$, $\F^c \in \FFF$ and for any $y \in \R^d$ one has $\Phi_y^c\in\PPPhi_y$.
\end{proposition}

\begin{proof}
This proof uses the concept of convolution on $X$, see e.g.\ \cite[p. 232]{Jikov1994} or \cite[Lemma 4.4]{Ruess2013}. We need a `smoothing kernel' $\kappa \in C^\infty_c$ which is assumed to be an even function, $\kappa \geq 0$, and $\int_{\R^d} \kappa(x) dx =1$. We rescale $\kappa_\epsilon(x):= \epsilon^{-d} \kappa(x/\epsilon)$ for $\epsilon >0$.

We start by proving $\F^c \in \FFF$: Let $f \in \F_w$. Without restriction assume $f \leq \|f\|_\infty$ and $\inf_\Omega f >0$. Choose $\delta$ s.t.\ $\inf_\Omega f>\delta>0$. With \cite[(4.12),(4.14)]{Ruess2013} choose $\epsilon>0$ s.t.\ 
\begin{align}\label{eq:fn_approximates1}
& \|f \ast \g_\epsilon - f\|_\nabla \leq \delta/3.
\end{align}
Define $d_\epsilon:= \sup_i \int |\partial_i \g_\epsilon |d\leb$ and set $\delta_\epsilon:= \delta / (1 \vee d_\epsilon) \leq \delta$. By Lusin's Theorem there exists a sequence of compact sets $K_n \subset \Omega$, $n \in \N$, s.t.\ $f$ is continuous on $K_n$ for $n \in \N$ and $\P[K_n] \nearrow 1$ for $n \to \infty$, see e.g.\ \cite[Theorem 7.1.13]{Bogachev2007}. The function $f|_{K_n}$ can be extended from the compact set $K_n$ to a continuous function $g_n$ on whole $\Omega$ s.t.\ $g_n|_{K_n} = f|_{K_n}$ and $\inf_\Omega f \leq g_n \leq \|f\|_\infty$ as is is stated for completely regular Hausdorff spaces in \cite[Exercise 6.10.22]{Bogachev2007}. Choose $n_0 \in \N$ s.t.\ for $n \geq n_0$,
\begin{align*}
\P[K_n^c] \leq \delta_\epsilon (3  \|f\|_\infty)^{-1},
\end{align*}
where $K_n^c:=\Omega \setminus K_n$. Let $a_n:= 1- \E g_n = \E[f-g_n]$. For $n \geq n_0$,
\begin{align*}
|a_n| \leq \E[|f-g_n|,K_n^c] \leq \|f\|_\infty \P[K_n^c] \leq \delta_\epsilon/3.
\end{align*}
We set 
\begin{align*}
f_n:= g_n +a_n.
\end{align*}
Then $\E[ f_n ] = 1$. Moreover, since $\delta < \inf_\Omega f$ one has $\inf_{n \geq n_0} f_n \geq \inf_\Omega f +a_n \geq \delta/2 >0$. And also $f_n \ast \g_\epsilon \in \F$, use e.g.\ \cite[Lemma 4.4]{Ruess2013}.

Moreover, $\omega \mapsto f_{n,\omega} (x) \g_\epsilon(x)$ is continuous and bounded by $\|f_n\|_\infty \|\g_\epsilon\|_\infty$ for any $x$. Since $\Omega$ is first countable, continuity is equivalent to sequential continuity, see \cite[Corollary 10.5]{Willard1970}. Hence Lebesgue's dominated convergence theorem may be applied in order to show that $f_n \ast \g_\epsilon$ is continuous in $\omega$. ('continuity of integrals with respect to a parameter', see e.g.\ \cite[Lemma 16.1]{Bauer2001}). A similar argument together with equality $\partial_i (f_n\ast \g_\epsilon) = - f_n \ast (\partial_i \g_\epsilon)$ shows that $D^m f_{n,\epsilon}$ is continuous and bounded for any $m \in \N_0$. In particular, $f_n \ast \g_\epsilon \in \F^c$.

$f_n \ast \g_\epsilon$ approximates $f$: Indeed, for $n \geq n_0$,
\begin{align}\label{eq:fn_approximates2}
\|f - f_n \|_2 
&= \|f-g_n-a_n\|_2
\leq \|f-g_n\|_2 + |a_n| \leq \|f\|_\infty \P[K_n^c] + \delta_\epsilon/3 \leq 2 \delta_\epsilon/3.
\end{align}
Further, Young's inequality, see e.g.\ \cite[(4.11)]{Ruess2013}, gives
\begin{align}\label{eq:fn_approximates3}
\|f \ast \g_\epsilon -f_n \ast \g_\epsilon\|_2
\leq \|f-f_n\|_2.
\end{align}
By \eqref{eq:fn_approximates1}, \eqref{eq:fn_approximates2}, \eqref{eq:fn_approximates3} we get
\begin{align*}
\|f -f_n \ast \g_\epsilon\|_2
\leq \|f - f\ast \g_\epsilon\|_2 + \|f\ast \g_\epsilon -f_n \ast \g_\epsilon\|_2 \leq  \delta.
\end{align*}
We consider derivatives in an analogous manner: Again with Young's inequality
\begin{align}\label{eq:fn_approximates4}
\|\partial_i (f \ast \g_\epsilon) & - \partial_i (f_n \ast \g_\epsilon)\|_2
= \|f \ast \partial_i (\g_\epsilon) - f_n \ast \partial_i (\g_\epsilon)\|_2
\leq d_\epsilon \|f - f_n\|_2.
\end{align}
Hence \eqref{eq:fn_approximates1}, \eqref{eq:fn_approximates2}, \eqref{eq:fn_approximates4} imply
\begin{align*}
\|\partial_i f  - \partial_i (f_n \ast \g_\epsilon)\|_2
& = \|\partial_i f - \partial_i (f \ast \g_\epsilon) \|_2 + \|\partial_i (f \ast \g_\epsilon)  - \partial_i (f_n \ast \g_\epsilon)\|_2\\
& \leq \delta/3 + d_\epsilon 2 \delta_\epsilon/3 \leq \delta.
\end{align*}
This proves $\F^c \in \FFF$.

In order to show $\D^c \in \DDD$ note first, that it is sufficient to show $\D^c$ dense in $\D$ since $\D \subset \D_w$ in the desired way by \cite[Lemma 2.1]{Ruess2013}.
Let $w \in \D$, $w \neq 0$ and consider $\psi:=(w - \E w)/(2\|w - \E w\|_\infty) +1$. $\psi \in \F$ and we can apply the previous and get a sequence $\psi_n \to \psi$ in $\|\cdot\|_\nabla$, $(\psi_n)_n \subset \F^c$. Then $w_n:= (\psi_n -1)2\|w-\E w\|_\infty + \E w \to w$ in the desired way and $(w_n)_n \subset \D^c$. Thus, $\D^c \in \DDD$.

In order to examine $\Phi_y^c$, since $\D^c$ is dense in $L^2$ s.t.\ $\partial_i \D^c \subset \D^c$ and $\tau_x \D^c \subset \D^c$, we may apply \cite[Lemma 4.7]{Ruess2013}. Using the fact that the space of weak divergence-free vector fields with expectation $y$ equals $\Phi_y^w$, use \cite[(4.18)]{Ruess2013}, we get $(\D^c)^d\cap \Phi_y^w$ is dense in $\Phi_y^w$ with respect to $\|\cdot\|_2$. Any $\phi \in (\D^c)^d\cap \Phi_y^w$ equals up to an exceptional set some $\tilde \phi \in (\D^c)^d\cap\Phi_y$. This shows $\Phi_y^c \in \PPPhi_y$.
\end{proof}

\subsection{Semi-continuity}\label{subsec:semicontinuity}

Our first continuity result considers also weak $L^1$ convergence:

\begin{proposition}\label{prop:application_continuity}
Let $V$, $V_n$, $n \in \N$, be potentials on $\Omega$. Assume that for all $f \in \F$ one has $\limsup_{n \to \infty} \E[V_nf] \leq \E[Vf]$, then for any $y \in \R^d$,
\begin{align}\label{eq:semicont_basic}
\limsup_{n \to \infty} \Gamma_{V_n}(y) \leq \Gamma_V(y).
\end{align}
\end{proposition}

Note that as soon as $V_n$, $n \in \N$, and $V$ are potentials in $L^2$, in order to obtain \eqref{eq:semicont_basic} it suffices to know that there exists a set $\tilde \F \in \FFF$ such that for all $f \in \tilde \F$ the condition $\limsup_{n \to \infty} \E[V_nf] \leq \E[Vf]$ is satisfied, use e.g.\ \cite[Proposition 2.2]{Ruess2013}.

\begin{proof}
By definition $\limsup_{n \to \infty} \Gamma_{V_n}(y)$ $=$ $\inf_{n \geq 0} \sup_{m \geq n} \Gamma_{V_m}(y)$. After an interchange of $\inf_{n \geq 0} \sup_{m \geq n}$ and $\inf_{f \in \F} \inf_{\phi \in \Phi_y}$ in the variational expression the statement follows.
\end{proof}

In order to study continuity with respect to weak convergence of the underlying probability measure, assume $\Omega$ is a topological space: Recall the condition (T) introduced on page \pageref{item:condT}.

\begin{theorem}\label{prop:application_weakcontinuity}
Assume $(\Omega,\FF,\P,\tau)$ satisfies \condT{} and $V$ is a potential, which is bounded and continuous with respect to the topology on $\Omega$.
Let $(\P_n)_n$ be a sequence of Radon probability measures on $(\Omega,\FF)$ such that $(\Omega,\FF,\P_n,\tau)$ is a metric dynamical system for all $n \in \N$. If $\P_n \to \P$ weakly, then for any $y \in \R^d$, with obvious notation,
\[
\limsup_{n \to \infty} \Gamma_{V}^{\P_n}(y) \leq \Gamma_V^{\P}(y).
\]
\end{theorem}


\begin{proof}
Let $\F_n^c$ and $\Phi_{y,n}^c$ denote the function spaces with respect to $\P_n$. We denote with $\E_n$ the expectation operator with respect to $\P_n$. Then one has bijective mappings
\begin{align*}
\F^c \to \F_n^c &: \; f \mapsto \tilde f:=f/\E_n[f], \text{ and }
\Phi_y^c \to \Phi_{y,n}^c : \phi \mapsto \tilde \phi := \phi - \E_n\phi +y.
\end{align*}
Therefore,
\begin{align*}
\limsup_{n \to \infty} \Gamma_{V}^{\P_n}(y)
&\leq 2 \limsup_{n \to \infty} \inf_{f \in \F_n^c} \inf_{\phi \in \Phi_{n,y}^c} \left(\E_n\left[\frac{|\nabla f|^2}{8f} + Vf\right]\,\E_n\left[\frac{|\phi|^2}{2f}\right]\right)^{1/2}\\
&= 2 \limsup_{n \to \infty} \inf_{f \in \F^c} \inf_{\phi \in \Phi_{y}^c} \left(\E_n\left[\frac{|\nabla f|^2}{8f} + Vf\right]\,\E_n\left[\frac{|\phi - \E_n \phi + y|^2}{2f}\right]\right)^{1/2}
\end{align*}
Similar as in the proof of Proposition \ref{prop:application_continuity}, the latter is lower or equal
\begin{align}\label{eq:weakcontinuity0}
2 \inf_{f \in \F^c} \inf_{\phi \in \Phi_y^c} \left(\inf_{n \geq 0}\sup_{m \geq n}\E_m\left[\frac{|\nabla f|^2}{8f} + Vf\right]\,\E_m\left[\frac{|\phi- \E_m \phi + y|^2}{2f}\right]\right)^{1/2}.
\end{align}
$\nabla f$ is continuous and bounded for $f \in \F^c$. So is $Vf$ by assumptions on $V$. Thus, weak convergence of $\P_n$ to $\P$ implies for $f \in \F^c$,
\begin{align}\label{eq:weakcontinuity1}
\E_n\left[\frac{|\nabla f|^2}{8f} + Vf\right] \to \E\left[\frac{|\nabla f|^2}{8f} + Vf\right] \text{ as } n \to \infty.
\end{align}
Again weak convergence shows for $\phi \in \Phi_y^c$ that $\E_n \phi \to \E \phi$ for $n \to \infty$. Therefore, $\E_n[|y-\E_n \phi|^2/(2f)] \leq (2\min_\Omega f)^{-1}|y-\E_n\phi|^2 \to 0$, and we get for $n \to \infty$,
\begin{align}\label{eq:weakcontinuity2}
\E_n\left[\frac{|\phi- \E_n \phi + y|^2}{2f}\right]
& =\E_n\left[\frac{|\phi|^2}{2f} + \frac{|y-\E_n \phi|^2}{2f} + \frac{2\phi \cdot ( y - \E_n \phi)}{2f}\right]
\to \E\left[ \frac{|\phi|^2}{2f}\right].
\end{align}
By Lemma \ref{lem:spacesdense_continuous} and \cite[Proposition 2.2]{Ruess2013} we can substitute the spaces $\F^c$ and $\Phi_y^c$ with the spaces $\F$, $\Phi_y$ in the definition of $\Gamma_V$, and we get with \eqref{eq:weakcontinuity0}, \eqref{eq:weakcontinuity1}, \eqref{eq:weakcontinuity2},
\begin{align*}
\limsup_{n \to \infty} \Gamma_{V}^{\P_n}(y)
\leq 2 \inf_{f \in \F^c} \inf_{\phi \in \Phi_y^c} \left( \E\left[\frac{|\nabla f|^2}{8f} + Vf\right] \, \E\left[\frac{|\phi|^2}{2f}\right]\right)^{1/2}
= \Gamma_V^{\P}(y),
\end{align*}
which was to be shown.
\end{proof}

\subsection{Scaling}

The variational formula also enables to determine convergence rates if scaled potentials are considered:

\begin{proposition}\label{prop:application_convergencespeed}
Let $c\geq 0$ and $V$ be a potential. Let $V_n:=V/n$. Then for all $y \in \R^d$,
\begin{align}\label{e:convspeed}
\Gamma_{V}^2(y) \leq n(\Gamma_{c+V_n}^2(y) - \Gamma_c^2(y)) \leq 2\E[V]|y|^2.
\end{align}
\end{proposition}

The rate of convergence as in Proposition \ref{prop:application_convergencespeed} for scaled potentials has been investigated previously in the discrete space setting of random walk in i.i.d.\ integrable potential in \cite{Wang02} and \cite{Kosygina10}. If $V$ is not necessarily integrable the asymptotic behaviour has been recently established in the discrete setting in \cite{Mountford2012} and \cite{Mountford2014}. For Brownian motion in Poissonian potential convergence speed is established in \cite{Ruess2012}. In \cite{Kosygina10, Mountford2012, Ruess2012} $c$ is assumed to equal zero. In \cite{Kosygina10, Ruess2012} the speed of convergence to zero of $\alpha_{V_n}$ has been determined to equal $n^{-1/2} \sqrt{2 \E[V]}|y|$. This coincides with the convergence speed $n^{-1/2}$ obtained from \eqref{e:convspeed} for $c=0$.

Additional assumptions allow to improve these results. For periodic potentials in Theorem \ref{prop:continuityonthetorus_speedofconvergence} we get exact rates of convergence for more general scalings of the potential. That Proposition \ref{prop:application_convergencespeed} is essentially all one might expect in general is illustrated by an example given in Subsection \ref{subsec:counterex:scaling}.

\begin{proof}[Proof of Proposition \ref{prop:application_convergencespeed}]
One has
$
n^{-1} \Gamma_{nV_n}^2(y) \leq \Gamma_{c+V_n}^2(y) - \Gamma_c^2(y) \leq 2\E[V_n]|y|^2,
$
where the upper bound follows from \eqref{eq:GammaV_leq_GammaEV} and \eqref{e:lyapunovExponent_forconstpot}, the lower bound from \eqref{e:c+Vgeq} and \eqref{e:cV}. Since $nV_n=V$ this shows the statement.
\end{proof}

\subsection{Continuity on the Torus}\label{subsec:semicontinuity_compact}

The results obtained in Subsection \ref{subsec:semicontinuity} can be improved considerably if the underlying space $\Omega$ is assumed to be compact: In the case that $X=X^{\tor{1}}$ where $\Omega$ is the one dimensional torus, see page \pageref{app:subsec:examples}, we get the following. We abbreviate for $f \in \F_w$,
\begin{align*}
B(f):=\inf_{\phi \in \Phi_y}\int \frac{|\phi|^2}{2f}d\P.
\end{align*}

\begin{theorem}\label{prop:continuityonthetorus}
Let $X=X^{\tor{1}}$. Let $V_n$, $n \in \N$, and $V$ be potentials such that $V_n \to V$ in $L^1$ and $V \geq \vmin>0$. Then there is a constant $C>0$, depending only on $\E[V]$ and $\vmin$, and there is $n_0 \in \N$ such that for $n \geq n_0$, for $y\in\R^d$,
\begin{align*}
|\Gamma_{V_n}^2(y)-\Gamma_V^2(y)| \leq C \|V_n -V\|_1|y|^2.
\end{align*}
\end{theorem}

\begin{proof}
Without restriction we may assume $|y|=1$, see \eqref{app:e:linearity}. Let $\epsilon_n:=\|V_n-V\|_1$ and choose $n_0$ such that for $n \geq n_0$,
\begin{align*}
\epsilon_n \leq (\vmin/2)(\sqrt{32 \E[V]}+1)^{-1}.
\end{align*}
Note that in particular, $\epsilon_n \leq \E[V]$.

One has by \eqref{eq:GammaV_leq_GammaEV} and \eqref{e:lyapunovExponent_forconstpot} for $n \geq n_0$,
\begin{align}\label{eq:GeneralBound}
\Gamma_{V_n}^2(y) \leq 4 \E[V] =: C_0.
\end{align}
We choose a `minimising' sequence $(f_n)_n \subset \F$ such that for $n \in \N$,
\begin{align}\label{eq:ChoiceOfSequence}
\Gamma_{V_n}^2(y) \geq 4 \E\left[ \frac{|f_n'|^2}{8f_n} + V_nf_n \right]B(f_n) - \epsilon_n.
\end{align}
An application of the `inverse' Hölder inequality \eqref{eq:holder_inverse} with $r=2$ shows
\begin{align}\label{eq:N_Bounded}
\E\left[\frac{|f_n'|^2}{f_n}\right] \geq \E[|f_n'|]^2,
\end{align}
since $\E[f_n]=1$. For $n \geq n_0$, by \eqref{eq:GeneralBound}, \eqref{eq:ChoiceOfSequence}, \eqref{eq:N_Bounded}, \eqref{eq:B_lowerbound}, since $\E[V_nf_n] \geq 0$, 
\begin{align}\label{eq:Nabla_Bound}
C_1:=(8C_0)^{1/2} \geq \E[|f_n'|].
\end{align}
An application of the fundamental theorem of calculus shows for $n \geq n_0$, for all $x< y \in \T^1$,
\begin{align}\label{eq:fundamentaltheoremofcalc}
|f_n(y)-f_n(x)|
& =|\int_x^y f_n'(t)dt|
\leq \int_x^y |f_n'(t)|dt
\leq C_1.
\end{align}
$\E[f_n]=1$, thus, each $f_n$ attains the value $1$. We get for $n$ large, $f_n(x) \leq C_1+1=:C_2$ for all $x \in \T^1$. Therefore, for $n \geq n_0$,
\begin{align}\label{eq:Vnfn-Vfn_leq}
|\E[V_nf_n]-\E[Vf_n]|\leq C_2 \|V_n-V\|_1.
\end{align}
We need an upper bound on $B(f_n)$: By \eqref{eq:GeneralBound}, \eqref{eq:ChoiceOfSequence} and \eqref{eq:Vnfn-Vfn_leq}, for $n \geq n_0$,
\begin{align*}
2C_0 \geq C_0 + \epsilon_n \geq 4\E[V_nf_n]B(f_n)
\geq 4 (\E[Vf_n]-C_2\epsilon_n) B(f_n)
\geq 2\vmin B(f_n).
\end{align*}
This shows that for $n \geq n_0$,
\begin{align}\label{eq:B_bounded}
B(f_n) \leq C_0/\vmin =: C_3.
\end{align}
Finally, by \eqref{eq:ChoiceOfSequence}, \eqref{eq:Vnfn-Vfn_leq}, \eqref{eq:B_bounded}, for $n \geq n_0$
\begin{align*}
\Gamma_{V_n}^2(y) 
& \geq 4 \E\left[ \frac{|f_n'|^2}{8f_n} + Vf_n \right]B(f_n) - 4C_2C_3\epsilon_n - \epsilon_n
\geq \Gamma_V^2(y) - (1+4C_2C_3)\epsilon_n.
\end{align*}

The proof of the upper bound is similar: Choose a minimising sequence $(g_n)_n \subset \F$ such that for $n \in \N$,
\begin{align}\label{eq:ChoiceOfSequence1}
\Gamma_V^2(y) \geq 4 \E\left[ \frac{|g_n'|^2}{8g_n} + Vg_n \right]B(g_n) - \epsilon_n.
\end{align}
As in \eqref{eq:Nabla_Bound} by \eqref{eq:GammaV_leq_GammaEV} and \eqref{e:lyapunovExponent_forconstpot}, `inverse' Hölder inequality, for $n \geq n_0$,
\begin{align*}
\E[|g_n'|] \leq C_1.
\end{align*}
Thus, as in \eqref{eq:fundamentaltheoremofcalc} for $n \geq n_0$, for $x \in \T^1$ one has $g_n(x)\leq C_2$. This shows
\begin{align}\label{eq:Vngn-Vgn_leq}
|\E[V_ng_n] - \E[Vg_n]| \leq C_2\|V_n-V\|_1.
\end{align}
We have similar to \eqref{eq:B_bounded} $2C_0 \geq 4 \vmin{} B(g_n)$, in particular,
\begin{align}\label{eq:B_bounded1}
B(g_n) \leq C_3.
\end{align}
Therefore, by \eqref{eq:Vngn-Vgn_leq}, \eqref{eq:B_bounded1} and \eqref{eq:ChoiceOfSequence1}, for $n \geq n_0$,
\begin{align*}
\Gamma_{V_n}^2(y)
& \leq 4 \E\left[ \frac{|g_n'|^2}{8g_n} + V_ng_n \right]B(g_n) \leq 4 \E\left[ \frac{|g_n'|^2}{8g_n} + Vg_n \right]B(g_n) + 4C_2C_3\epsilon_n\\
& \leq \Gamma_V^2(y) + (1+4C_2C_3)\epsilon_n.
\end{align*}
This shows the statement.
\end{proof}

As we have an $L^1$-Poincar\'e inequality on the $d$-dimensional torus, we can calculate the convergence rate on the torus exactly:

\begin{theorem}\label{prop:continuityonthetorus_speedofconvergence}
Let $X=X^{\tor{d}}$ and $V_n$, $n \in \N$, $V$ be potentials. Assume $nV_n \to V$ for $n \to \infty$ in $L^1$ and $V$ is bounded. Let $c \geq 0$, then for $y \in \R^d$,
\begin{align*}
n(\Gamma_{V_n+c}^2(y)-\Gamma_c^2(y)) \to 2\E[V]|y|^2 \text{ as } n \to \infty.
\end{align*}
\end{theorem}

\begin{proof}
Let $y\neq 0$. The upper bound follows from \eqref{eq:GammaV_leq_GammaEV}, \eqref{e:lyapunovExponent_forconstpot}. For the lower let $n_0 \in \N$ such that for $n \geq n_0$ one has $\E[nV_n] \leq 2\E[V]$. By \eqref{eq:GammaV_leq_GammaEV}, \eqref{e:lyapunovExponent_forconstpot} for $n \geq n_0$,
\begin{align}\label{eq:nGammaVn_Bounded}
\psi_n:=n\Gamma_{V_n}^2(y) \leq 2\E[nV_n]|y|^2 \leq 4 \E[V]|y|^2=:C_0.
\end{align}
Choose $(f_n)_n \subset \F$ such that
\begin{align}\label{eq:ChoiceOfSequence2}
\psi_n
& = 4 n \inf_{f \in \F} \E\left[\frac{|\nabla f|^2}{8f} + Vf\right]B(f)
\geq 4 n \E\left[\frac{|\nabla f_n|^2}{8f_n} + Vf_n\right]B(f_n)-1/n.
\end{align}
Therefore, with  \eqref{eq:nGammaVn_Bounded} and \eqref{eq:B_lowerbound}, for $n \geq n_0$,
$
(C_0+1/n)/n \geq \psi_n/n \geq 2\E[|\nabla f_n|^2/(8f_n)]|y|^2,
$
which shows
\begin{align}\label{eq:N_ToZero}
\E\left[\frac{|\nabla f_n|^2}{8f_n} \right] \to 0 \text{ as } n \to \infty.
\end{align}
Using `inverse' Hölder inequality \eqref{eq:holder_inverse} and Poincar\'e inequality, see \cite[(7.45)]{Gilbarg83}, we get for $n \in \N$,
\begin{align*}
\E\left[\frac{|\nabla f_n|^2}{8f_n} \right] 
\geq \E[|\nabla f_n|]^2/8
\geq c_p\E[|f_n-1|]^2/8.
\end{align*}
where the constant $c_p$ comes from the Poincar\'e inequality. Thus, by \eqref{eq:N_ToZero}
\begin{align}\label{eq:f_n_to_zero}
\|f_n -1\|_1 \to 0 \text{ as } n \to \infty.
\end{align}
In particular, the set $\{f_n: \, n \in \N\}$ is uniformly integrable, see \cite[Theorem 4.5.2]{Durrett96Probability}, and we get for $M_n:=\|nV_n-V\|_1^{-1/2}$ that
\begin{align}\label{eq:fn,fn_geq_Mn}
\epsilon_{1,n} := \E[f_n,f_n \geq M_n]  \to 0 \text{ as } n \to \infty.
\end{align}
We may estimate for $n \in \N$,
\begin{align}\label{eq:EnV_n_geq}
\E[nV_nf_n]
&\geq \E[V(f_n\wedge M_n)] - \|nV_n-V\|_1M_n\nonumber\\
&\geq \E[Vf_n] - \|V\|_\infty \E[|f_n - f_n\wedge M_n|] - \|nV_n-V\|_1^{1/2}\nonumber\\
&= \E[Vf_n] - \|V\|_\infty\epsilon_{1,n} - \|nV_n-V\|_1^{1/2}
\geq \E[V] - \epsilon_{2,n},
\end{align}
with $\epsilon_{2,n}:=\|V\|_\infty(\|f_n-1\|_1 + \epsilon_{1,n}) + \|nV_n-V\|_1^{1/2}$. Note that by \eqref{eq:f_n_to_zero}, \eqref{eq:fn,fn_geq_Mn} and by assumptions on $V$ one has $\epsilon_{2,n} \to 0$ as $n \to \infty$. We need control of $B(f_n)$: Let $n_1 \geq n_0$ such that for $n \geq n_1$ one has $1/n \leq C_0$ and $\epsilon_{2,n} \leq \E[V]/2$. Then with \eqref{eq:nGammaVn_Bounded}, \eqref{eq:ChoiceOfSequence2}, \eqref{eq:EnV_n_geq} for $n \geq n_1$
\begin{align*}
2C_0\geq C_0 + 1/n\geq 4 \E[nV_nf_n]B(f_n)
\geq 4(\E[V]-\epsilon_{2,n})B(f_n)
\geq 2\E[V]B(f_n).
\end{align*}
This shows for $n \geq n_1$,
\begin{align}\label{eq:B_bounded2}
B(f_n) \leq C_0/\E[V]=: C_1.
\end{align}
Therefore, by \eqref{e:c+Vgeq}, \eqref{eq:ChoiceOfSequence2}, \eqref{eq:EnV_n_geq}, \eqref{eq:B_bounded2} and \eqref{eq:B_lowerbound}, for $n \geq n_1$,
\begin{align*}
n(\Gamma_{c+V_n}^2(y)-\Gamma_c^2(y))
& \geq n\Gamma_{V_n}^2(y)
\geq 4 n \E\left[\frac{|\nabla f_n|^2}{8f_n} + V_nf_n\right]B(f_n)-1/n\\
& \geq 4 (\E[V]-\epsilon_{2,n})B(f_n)-1/n
 \geq 4 \E[V]B(f_n)- 4C_1\epsilon_{2,n} -1/n\\
& \geq 2\E[V]|y|^2- 4C_1\epsilon_{2,n} -1/n.
\end{align*}
This finishes the argument.
\end{proof}

\section{Examples}\label{sec:counterex}

The Lyapunov exponent is semi-continuous in many cases as outlined in Section \ref{sec:continuity}. We provide examples which show that continuity of the Lyapunov exponent with respect to weak convergence of the underlying measure, continuity with respect to $L^p$ convergence of the potential, $1 \leq p < \infty$, and also a convergence speed as the one established in Proposition \ref{prop:application_convergencespeed} are not valid in general. This should be compared to similar models such as random walks in random potential, and we refer to the discussion given in Section \ref{sec:continuity}.

The example we present is built on homogeneous Poisson line processes. In particular, the underlying probability measure is isotropic, whereas it does not satisfy a `finite range dependence property'.
We start recalling Poisson line processes and refer to \cite[Section 15.3]{Daley2008}, \cite[Chapter 8]{Stoyan1987} for more detailed descriptions.

\subsection{The Poisson Line Process}

Let $e_1$ and $e_2$ denote the unit vectors in $\R^2$. Any (undirected) line $\ell$ in $\R^2$ can be represented by its angle $\theta$ with a reference line and its (signed) distance $r$ to a reference point. We choose as reference line the $x_1$-axis and as reference point the origin. The angle is measured starting from the $x_1$-axis anticlockwise. The distance $r$ is chosen to be nonnegative if $\ell$ intersects $\{te_2: \; t\geq 0\}$ or if $\ell$ is parallel to $e_2$ intersecting $\{te_1:\; t >0\}$. Else, $r$ is chosen negative. This leads to a bijective correspondence $\rho: \LL \to \rep$ between the set $\LL$ of lines in $\R^2$ and and the `representation space' $\rep:=\R \times (0,\pi]$.  If $\ell = \rho^{-1}(r, \theta)$ we also simply write $\ell = (r,\theta)$. Let $\BB(\rep)$ denote the Borel $\sigma$-algebra on $\rep$.

Let $\Omega$ be the set of locally finite measures on $(\rep, \BB(\rep))$ equipped with the topology of vague convergence and let $\FF$ be the Borel $\sigma$-algebra on $\Omega$. We introduce an action of $\R^2$ on $\Omega$ in the following way: $(\R^2,+)$ is acting on $\LL$ via $\tau_x^\LL:$ $\ell\mapsto \ell + x$, where $x \in \R^2$, $\ell \subset \LL$. This induces an action of $(\R^2,+)$ on $\rep$ given by $\tau_x^\rep:(r,\theta) \mapsto \rho( \tau^\LL_x(\rho^{-1}(r,\theta)))$, where $(r,\theta)\in \rep$, $x \in \R^2$. Finally we introduce the action of $(\R^2,+)$ on $\Omega$ as $\tau_x: \Omega \to \Omega$, $\tau_x \omega[A] := \omega[\tau_x^\rep A]$, where $A \in \BB(\rep)$, $\omega \in \Omega$ and $x \in \R^2$. Note that the action $\tau^\rep$ is no simple shift on the cylinder, but a shear, see \cite[(8.2.1)]{Stoyan1987} or \cite[(15.3.1)]{Daley2008} where formulae for directed lines are given. 
The continuity properties of $\tau^\rep_\cdot \cdot$ obtained from such formulae ensure that $\tau$ is product measurable analogous to \cite[Exercise 12.1.1(a)]{Daley2008}.

The (homogeneous) Poisson line process is given by the representation $\rho$ and the distribution $\P_\kappa$ of a Poisson point process on $(\rep, \BB(\rep))$ having intensity measure $\nu=\kappa\cdot\leb\otimes\mu$ with $\mu$ the uniform distribution on $(0,\pi]$ and $\kappa>0$. The tupel $(\Omega,\FF,\P_\kappa,\tau)$ is an ergodic dynamical system, as outlined for example in \cite[p.\ 99]{Cowan1980} and \cite[Theorem 1]{Miles1964}. Moreover, it is isotropic, see e.g.\ \cite[p.\ 902]{Miles1964}.

\subsection{Discontinuity with Respect to the Underlying Measure}\label{subsec:counterx:continuity_weak_prob}

Some additional notation is needed: Let $R>0$, $x \in \R^2$, and let $\ell$ be a line in $\R^2$. We denote by $B_R(x)$ the closed ball with centre $x$ and radius $R$, and we introduce stripes $Q_R(\ell)$ given by 
\begin{align*}
Q_R(\ell):=\{y \in \R^2:\; d(y,\ell) < R \}.
\end{align*}
By $H_R(x)$ we denote the entrance time of $Z$ into $B_R(x)$, that is $H_R(x):= \inf\{t \geq 0: \; Z_t \in B_R(x)\}$. We introduce the exit time of $Z$ from $Q_R(\ell)$ by $\tau_R(\ell):= \inf\{t \geq 0: \; Z_t \notin Q_R(\ell)\}$. Then $H_R(x)$ and $\tau_R(\ell)$ are stopping times with respect to the canonical filtration of $(Z_t)_t$, see e.g.\ \cite[Problem 1.2.7]{Karatzas1991}.

The potential we consider is defined as follows: For $\omega \in \Omega$ let $[\omega]$ be the support of $\omega$. If $F$ is a subset of $\R^2$ and $\omega \in \Omega$ we introduce the intersection of $F$ with the lines of $\omega$ by  $[\omega] \cap F := \bigcup_{z \in [\omega]} (F \cap \rho^{-1}(z))$. Let $\constpot, M \geq 0$ and $R>0$. We define the potential $V:\Omega \to [0,\infty)$,
\begin{align}\label{eq:def_pot}
V(\omega):=V_{\constpot,R,M}(\omega):= \constpot + M\cdot 1_{\{\tilde \omega \in \Omega:\; [\tilde \omega] \cap B_R(0) = \emptyset\}}(\omega),
\end{align}
which equals $M+\constpot$ outside of stripes of radius $R$ along the lines of $\omega$, and which equals $\constpot$ inside these stripes.

We show the following: Let $\lambda_2$ be the principal Dirichlet eigenvalue of $-(1/2)\Laplace$ in the unit ball in $\R^2$.

\begin{theorem}\label{th:counterex}
For any $D>0$ there is $R_0>0$ such that for $\kappa>0$ one has $\P_\kappa$-a.s.\ for all $\constpot \geq 0$ and for all $y \in S^{d-1}$,
\begin{align}\label{eq:counterex}
\sup_{R \geq R_0} \sup_{M \geq 0} \limsup_{u \to \infty}
- \frac{1}{u}\ln E_0[\exp\{-\int_0^{H_1(uy)}  (V_{\constpot, R,M})_\omega(Z_s)ds\}] 
\leq \sqrt{2\constpot}+ D.
\end{align}
We may choose $R_0=4\sqrt{\lambda_2}/D+1$.
\end{theorem}

Recall, that by \eqref{e:lyapunovExponent_forconstpot} the right side of \eqref{eq:counterex} equals $\alpha_\constpot(y)+D$.

This result contradicts continuity of the Lyapunov exponent with respect to weak convergence of the underlying probability measure: The convolution with an even and smooth function $g:\R^2 \to [0,\infty)$ of support $\supp g \subset B_{R/2}(0)$ and $\int_{\R^2} g(x) dx=1$, see e.g.\ \cite[Lemma 4.4]{Ruess2013}, leads to a \textit{regular potential} $W:=V_{1,2R,1}\ast g \leq V_{1,R,1}$ for which the Lyapunov exponent exists and can be expressed as follows: $\P_\kappa$-a.s.\ the limit in the following exists and equals
\begin{align*}
\lim_{u \to \infty} - \frac{1}{u}\ln E_0[\exp\{-\int_0^{H_1(ue_1)} W_\omega(Z_s)ds\}] = \alpha_{W}^{\P_\kappa}(e_1),
\end{align*}
see \cite[(1.9)]{Ruess2013}. With Theorem \ref{th:counterex} for $D= (\alpha_2(e_1)-\alpha_1(e_1))/2$ there is $R>0$ such that
\begin{align}\label{eq:contradiction}
\sup_{\kappa>0} \alpha_W^{\P_\kappa}(e_1)
\leq \sup_{\kappa>0} \alpha_{V_{1,R,1}}^{\P_\kappa}(e_1)
\leq \alpha_1(e_1)+D < \alpha_2(e_1) =\alpha_W^{\delta_\textbf{0}}(e_1),
\end{align}
where $\textbf{0}$ is the zero measure on $\rep$. On the other hand note that $\P_{\kappa} \to \delta_\textbf{0}$ weakly as $\kappa \to 0$. Such convergence follows e.g.\ with help of Laplace transforms of Point processes, see \cite[(9.4.17), Theorem 11.1.VIII]{Daley2008}. This together with \eqref{eq:contradiction} shows discontinuity as stated.


We start with an estimate on the travel costs along stripes which is analogous to \cite[(5.2.32)]{Sznitman98}. Let $\ell_0$ denote the $x_1$-axis.

\begin{lemma}\label{lem:costsalongline}
Let $R>0$, $\constpot \geq 0$ and $u > R$. Then
\begin{align*}
&E_0[\exp\{-\constpot H_R(ue_1)\}, \tau_R(\ell_0)>H_R(ue_1)] 
\geq  C \exp\{-u\sqrt{2(\constpot+\lambda_2/R^2)}\},
\end{align*}
where $C>0$ is a constant.
\end{lemma}

\begin{proof}
For any $t > 0$, with Girsanov's formula, see \cite[Theorem 3.5.1, Corollary 3.5.13]{Karatzas1991},
\begin{align}
&E_0[\exp\{-\constpot H_R(ue_1)\}, \tau_R(\ell_0)>H_R(ue_1)]
\geq \exp\{-\constpot t\} P_0[\sup_{0\leq s \leq t}|Z_s-\frac{s}{t}ue_1|< R]\nonumber\\
& \quad = \exp\{-\constpot t\} E_0[\exp\{-\frac{u}{t}e_1\cdot Z_t - \frac{u^2}{2t}\}, \sup_{0\leq s \leq t}|Z_s|< R].\label{eq:linetravelcosts}
\end{align}
We abbreviate $\BB:=\{\sup_{0\leq s \leq t}|Z_s|< R\}$. Note that $E_0[Z_t|\BB]=0$, since $-Z \overset{d}{=} Z$ and $Z \in \BB$ if and only if $-Z \in \BB$. An application of Jensen inequality shows, that \eqref{eq:linetravelcosts} is greater or equal
\begin{align*}
&\exp\{-\constpot t\} \exp\{-\frac{u}{t}e_1\cdot E_0[Z_t|\BB] - \frac{u^2}{2t}\} P_0[ \BB]\\
&= \exp\{-\constpot t -\frac{u^2}{2t}\} P_0[ \BB]
\geq C \exp\{-\constpot t-\frac{u^2}{2t}-\lambda_2t/R^2\},
\end{align*}
where for the last estimate we used \cite[(3.1.53)]{Sznitman98}. The choice $t:=u/\sqrt{2(\constpot + \lambda_2/R^2)}$ shows the statement.
\end{proof}

In order to prove Theorem \ref{th:counterex} we need to construct a path such that travelling along this path is relatively cheap for the Brownian motion. Therefore, a great part of this path should lie in regions of low potential. This forces the path to follow closely the lines of the Poisson line process. On the other hand the path should not be too long. A path following mainly the lines and `exceeding' the Euclidean distance only logarithmically can be found in \cite{Aldous2008, Kendall2011}. For our purposes it suffices to find a path with linear `exceedance'.

\begin{proof}[Proof of Theorem \ref{th:counterex}]
For any direction $y \in S^{d-1}$ we need to have a `suitable' line leading into this direction: Let $\imaginary$ denote the complex number $(0,1)\in\C$. For $y \in S^{d-1}$, for $\psi \in (0,\pi/2)$ we introduce the event $\AAa(y,\psi) \in \FF$ consisting of those $\omega \in \Omega$ for which there is a line $\ell = (r,\theta) \in [\omega]$ such that the angle between $\ell$ and $y$, measured from $y$ anticlockwise, is in $[\pi-\psi,\pi)$, and $\ell \cap \{ e^{t\imaginary},t>0\} \neq \emptyset$. Set $\Omega_1:=\bigcap_{0<\psi<\pi/2} \bigcap_{y \in S^{d-1}} \AAa(y,\psi)$.

For all $\kappa>0$ one has $\P_\kappa[\Omega_1]=1$. In fact, for any $y \in S^{d-1}$ and $\psi \in (0,\pi/2)$ one has $\P_\kappa[\AAa(y,\psi)] = 1$. This can be verified by considering the distribution of the intersection points and angles of the lines $\ell \in \rho^{-1}[\omega]$ with a fixed line, see e.g.\ \cite[Theorem 2]{Miles1964}. Recognise, that for $\psi \in (0,\pi/2)$ and for $0 \leq t \leq \psi/4$ one has $\AAa(ye^{t\imaginary},\psi) \supset \AAa(y,\psi/4)$. Thus, if we divide the interval $[0,2\pi)$ into a finite number of intervals $I_k=[x_k,x_{k+1})$, $k=1, \ldots \bar k$ of same length $x_{k+1}-x_k\leq \psi/4$, $x_1=0$, $x_{\bar k + 1} = 2\pi$, we get 
$\P_\kappa [ \bigcap_{y \in S^{d-1}} \AAa(y,\psi)] 
\geq  \P_\kappa [ \bigcap_{k = 1 \ldots, \bar k} \AAa(e^{x_k\imaginary},\psi/4)]=1$.
For all $y \in S^{d-1}$ the events $\bigcap_{y \in S^{d-1}}\AAa(y,\psi)$ are monotone increasing as $\psi$ increases. Therefore, looking only at a countable number of angles $(\psi_n)_n$, $\psi_n \to 0$ as $n \to \infty$, we even have 
$\P_\kappa [ \Omega_1] 
=\P_\kappa [ \bigcap_{n \in \N} \bigcap_{y \in S^{d-1}} \AAa(y,\psi_n)]=1$. 

Let $\omega \in \Omega_1$. Let $D>0$ and define
\[
R_0:=4\sqrt\lambda_2/D + 1.
\]
Let $R \geq R_0$ and $\constpot, M \geq 0$, and let the potential $V=V_{c,R,M}$ be given as in \eqref{eq:def_pot}. Define $\middleconst_1:= \sqrt{2\constpot + 2\lambda_2/( R_0 - 1)^2}$, $\middleconst_2:=\alpha_\constpot(e_1) + D/2 $, and
\[
\varphi:= \min\{\arctan (D/(16\alpha_{\constpot + M}(e_1))),\arccos (\middleconst_1/\middleconst_2)\}.
\]
Note that by \eqref{e:lyapunovExponent_forconstpot} $0<\middleconst_1/\middleconst_2 <1$, and therefore $\varphi \in (0,\pi/2)$.

We restrict ourselves in the following to the case $y=e_1$. By rotation invariance of the law of Brownian motion, the same argument shows the statement for any $y \in S^{d-1}$.

Let $u \geq R+1$. We construct a path $\gamma$ starting in $0$ and leading to $ue_1$ as follows, see Figure \ref{fig:path}. We start the path in $0$ in direction of $e_2$ until hitting a line $\ell_\gamma=(r,\theta_\gamma)\in [\omega]$ of angle $\theta_\gamma \in [\pi-\varphi, \pi)$. Such a line exists by choice of $\omega \in \Omega_1 \subset \AAa(e_1,\varphi)$. We denote the intersection point by $p_1$. The path now follows the line $\ell_\gamma$ until the intersection of $\ell_\gamma$ with the line $\{ue_1 + s e_2 : \; s \in \R\}$. We denote this intersection point by $p_2$. Then the path follows this vertical line until hitting $ue_1$.

We divide the journey of the Brownian motion into three different parts, see also Figure~\ref{fig:path}: Define stopping times

\begin{align*}
\HHb := H_R(p_2) \circ \Theta_{\HHa} + \HHa,\
\HHc := H_1(ue_1) \circ \Theta_{\HHb} + \HHb,
\end{align*}
where $\Theta$ denotes the shift on the pathspace $\Sigma$, that is $\Theta_t((w_s)_{s\geq 0})=(w_{s+t})_{s \geq 0}$ for $w \in \Sigma$, $t \geq 0$. Let $\AAa:=\{ \tau_R(\ell_\gamma) \circ \Theta_{\HHa}> H_R(p_2) \circ \Theta_{\HHa}\}$.
We estimate and split the integral:
\begin{align}
&E_0[ \exp \{ -\int_0^{H_1(ue_1)} V_\omega(Z_s)ds\}]
\geq E_0[ \exp \{ -\int_0^{\HHc} V_\omega(Z_s)ds\}, \AAa]\nonumber\nonumber\\
&\qquad = E_0[ \exp \{ -\int_0^{\HHa} V_\omega(Z_s) ds - \int_{\HHa}^{\HHb} V_\omega(Z_s) ds - \int_{\HHb}^{\HHc} V_\omega(Z_s)ds\}, \AAa ]\label{eq:integral_devided}
\end{align}
The potential $V$ is bounded by $\constpot + M$, and on $\AAa$ for $\HHa \leq t \leq \HHb$ we have $V_\omega(Z_t)=\constpot$. All considered stopping times are $P_0$-a.s.\ finite. Thus, an application of the strong Markov property, see \cite[Theorem 2.6.15]{Karatzas1991}, shows that we can bound \eqref{eq:integral_devided} form below by
\begin{align}
& E_0[ \exp \{-(\constpot+ M) \HHa\}] 
\inf_{x \in  B_1(p_1)} E_x[\exp\{-\constpot H_R(p_2)\}, \tau_R(\ell_\gamma) > H_R(p_2)]\label{eq:appl_strongmarkprop}\\
&\qquad\cdot \inf_{x \in  B_R(p_2)} E_x[\exp\{-(\constpot+ M) H_1(ue_1)\}].\nonumber
\end{align}

\begin{figure}[h]
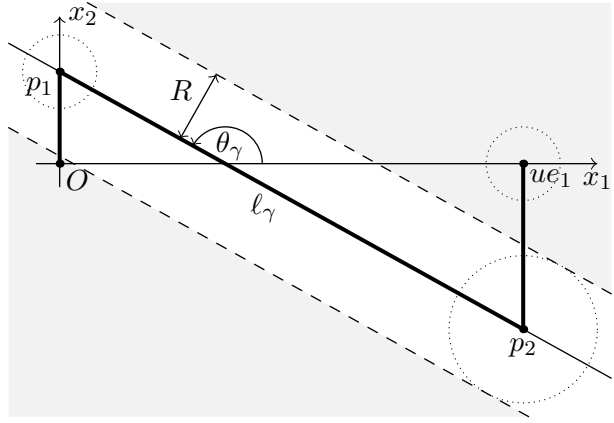

\includepathfigure
\caption{We see a small sector of $\R^2$. The only line of $\omega$ passing this sector is $\ell_\gamma$. The emphasised line segments illustrate the path $\gamma$ from $0$ to $ue_1$ through $p_1$ and $p_2$. $Z$ is observed until hitting $B_1(p_1)$, then until reaching $B_R(p_2)$ `forced' to stay in the region $Q_R(\ell_\gamma)$ of low potential, and thereafter until hitting $B_1(ue_1)$.}
\label{fig:path}
\end{figure}

We start estimating the middle term of \eqref{eq:appl_strongmarkprop}: Since $u \geq R+1$ we have $|p_2-p_1|=u/\cos(\pi-\theta_\gamma) \geq R+1$. For $x \in  B_1(p_1)$ we set $\bar x := x-p_1$. Since $R > 1$ we have $B_R(p_2) \supset B_{R-1}(p_2+\bar x)$ and $Q_R(\ell_\gamma) \supset Q_{R-1}(\ell_\gamma+\bar x)$. Therefore, for $x \in B_1(p_1)$,
\begin{align*}
&E_x[\exp\{-\constpot H_R(p_2)\}, \tau_R(\ell_\gamma) > H_R(p_2)]\\
&\qquad \geq E_x[\exp\{-\constpot H_{R-1}(p_2+\bar x)\}, \tau_{R-1}(\ell_\gamma+\bar x) > H_{R-1}(p_2+\bar x)].
\end{align*}
$\ell_\gamma + \bar x$ leads through $x$ and through $x + p_2-p_1$. The law of Brownian motion is invariant under translations and rotations. Thus Lemma \ref{lem:costsalongline} applied with radius $R-1$ shows for $R \geq R_0$ that the middle term in \eqref{eq:appl_strongmarkprop} can be bounded from below by
\begin{align}\label{eq:bound_along_lines}
C \exp\{-|p_2-p_1|\sqrt{2\constpot+2\lambda_2/(R_0-1)^2}\}.
\end{align}

In order to get a bound on the last term of \eqref{eq:appl_strongmarkprop} let $a_u:= u \tan(\pi - \theta_\gamma) - | p_1|$. Note that $0<\tan(\pi-\theta_\gamma) \leq \tan \varphi$, and $a_u \sim u\tan(\pi - \theta_\gamma)$ as $u \to \infty$. Note also that $a_u=|p_2-ue_1|$ if $a_u \geq 0$. Let $u_0 \geq R+1$ such that for $u \geq u_0$ one has $d(B_R(p_2),ue_1) \leq 2a_u$ and
\begin{align*}
- \frac{1}{u} \ln E_0[\exp\{-(\constpot + M)  H_1(2a_ue_1)\}]
\leq \frac{4a_u}{u}  \alpha_{\constpot + M}(e_1) \leq 8(\tan\varphi)\alpha_{\constpot + M}(e_1),
\end{align*}
where the first inequality is a consequence of the existence of the Lyapunov exponent for constant potential.  Then for $u \geq u_0$ the last term of \eqref{eq:appl_strongmarkprop} can be bounded from below by
\begin{align}\label{eq:choiceofn}
E_0[\exp\{-(\constpot+ M) H_1(2a_ue_1)\}]
\geq \exp\{- 8 u (\tan \varphi) \alpha_{\constpot + M}(e_1)\}.
\end{align}

Since the first term in \eqref{eq:appl_strongmarkprop} only depends on $\omega$, since $|p_2-p_1|=u/\cos (\pi-\theta_\gamma) \leq u/\cos \varphi$, we get with \eqref{eq:appl_strongmarkprop}, \eqref{eq:bound_along_lines} and \eqref{eq:choiceofn} for $R \geq R_0$ and for $u \geq u_0$,
\begin{align*}
& \limsup_{u \to \infty} - \frac{1}{u} \ln E_0[ \exp \{ -\int_0^{H_1(ue_1)} V_\omega(Z_s)ds\}]\\
& \qquad \leq \sqrt{2(\constpot+\lambda_2/(R_0-1)^2)}/\cos\varphi  + 8 (\tan\varphi) \alpha_{\constpot + M}(e_1)
\end{align*}
which is lower or equal $\alpha_\constpot(e_1)+D$ by the choice of $\varphi$.
\end{proof}

\subsection{Discontinuity with Respect to the Potential}

A slight modification of the previous setting also shows, that the Lyapunov cannot be continuous in general with respect to $L^p$-convergence of the potential, $1 \leq p < \infty$. Extend $\rep$ to $\exrep:=\rep\times[0,\infty)$. Let $\exOmega$ be the space of locally finite discrete measures on $\exrep$ provided with the topology of vague convergence and let $\exFF$ be the Borel $\sigma$-algebra. Let $\exP$ be the law of a homogeneous Poisson point process on $\exOmega$ with intensity measure $\leb \otimes \mu \otimes \leb$. Then $\R^2$ is acting on $\exrep$ via $\extau^{\exrep}_x: (r,\theta,s) \mapsto (\tau_x^\rep(r,\theta),s)$, where $x \in \R^2$. This again leads to an action $\extau$ of $\R^2$ on $\exOmega$ as before, under which $\exP$ is invariant and ergodic.

For $0<\kappa$ let $\Phi_\kappa$ be the mapping from $\exOmega$ to $\Omega$ defined as follows: We may represent discrete $\omega \in \exOmega$ as sums of Dirac measures: $\omega = \sum_{i \in \N} \delta_{(r_i,\theta_i,s_i)}$, see e.g.\ \cite[Proposition 9.1.III]{Daley2008}. We define for discrete $\omega \in \exOmega$ the mapping
\[
\Phi_\kappa: \, 
\omega=\sum_{i \in \N} \delta_{(r_i, \theta_i, s_i)} 
\mapsto 
\sum_{i\in \N: \, s_i < \kappa} \delta_{(r_i, \theta_i)}.
\]
Then $\Phi_\kappa(\exP) = \P_\kappa$.

We introduce for $\omega \in \exOmega$ and for $\kappa, R>0$ the potential
\[
\exV_{\kappa,R}(\omega):= V_{1,R,1} \circ \Phi_\kappa (\omega).
\]
For all $1 \leq p < \infty$ and $R>0$ we have
$
\exV_{\kappa,R} \to 2 \text{ in } L^p(\exP)
$
as $\kappa \to 0$. Indeed, 
\begin{align}
\exE[|2-\exV_{\kappa,R}|^p] 
&= \exP[\{\omega \in \exOmega:\,[\Phi_\kappa(\omega)]\cap B_R(0) \neq \emptyset\}]
= \P_\kappa[\{\omega \in \Omega:\,[\omega]\cap B_R(0) \neq \emptyset\}]\nonumber\\
&= \P_\kappa[\{\omega \in \Omega:\, \omega[[-R,R]\times(0,\pi]] \neq 0\}] 
= 1- e^{- 2 \kappa R} \to 0 \text{ as } \kappa \to 0,\label{eq:ex:counter2}
\end{align}
since for discrete $\omega \in \Omega$ one has $[\omega] \cap B_R(0) \neq \emptyset$ if and only if $\omega[[-R,R]\times(0,\pi]] \neq 0$.
On the other hand as after Theorem \ref{th:counterex} let $\exW_\kappa := (V_{1,2R,1} \circ \Phi_\kappa) \ast g$, then $\exW_\kappa$ is a \textit{regular potential} such that $\exW_\kappa \leq \exV_{\kappa,R}$. In order to clarify dependence on $\omega$ we introduce
\[
a(u,U,\omega):=-\ln E_0[\exp\{-\int_0^{H(ue_1)} U_\omega(Z_s)ds\}],
\]
where $u>0$, where $U$ is some potential on $\Omega$ or $\exOmega$, and where $\omega \in \Omega$ or $\omega \in \exOmega$ respectively. Then with \cite[(1.9)]{Ruess2013} and by Theorem \ref{th:counterex} applied to $D:=(\alpha_2(e_1)-\alpha_1(e_1))/2$, there is $R>0$ such that we have for $\kappa >0$ $\exP$-a.s.,
\begin{align*}
\alpha_{\exW_\kappa}(e_1)
& = \lim_{u \to \infty}\frac{1}{u}a(u,\exW_{\kappa},\omega) 
\leq \limsup_{u \to \infty}\frac{1}{u}a(u,\exV_{\kappa,R},\omega)\\
& = \limsup_{u \to \infty}\frac{1}{u}a(u,V_{1,R,1},\Phi_\kappa(\omega)) 
\leq \alpha_1(e_1) + D 
< \alpha_2(e_1).
\end{align*}
This, continuity of the convolution, see e.g.\ \cite[(4.11)]{Ruess2013}, and \eqref{eq:ex:counter2} show discontinuity as announced.

\subsection{Untypical Scaling.} \label{subsec:counterex:scaling}

The previous example can also be used, in order to show that in general convergence of a sequence of potentials $(V_n)_n$ to zero in the sense that there exists a potential $V$ with $nV_n \to V$ in $L^p$ for some $1 \leq p < \infty$ does not guarantee $\sqrt n \alpha_{V_n}(e_1) \to \sqrt{2\E V}$.

We consider $(\exOmega,\exFF,\exP,\extau)$ and define for $b,n \in \N$ the potential
\begin{align*}
\excontV_{n,b}:=\frac{1}{n} \exV_{1/n^2,bn} = V_{1/n,bn,1/n} \circ \Phi_{1/n^2}.
\end{align*}
Then for $1 \leq p < \infty$, for $b \in \N$, the potentials $n\excontV_{n,b}$ converge to $2$ in $L^p$ as $n \to \infty$. Indeed, as in \eqref{eq:ex:counter2},
\begin{align}
\exE[|n\excontV_{n,b} - 2|^p] 
&= \P_{1/n^2}[\{\omega\in \Omega:\; \omega[[-bn,bn]\times(0,\pi]]\neq 0\}] 
= 1-e^{-2bn/n^2} \to 0 \label{eq:counterex3}
\end{align}
as $n \to \infty$. On the other hand, set $\excontW_n:=\excontV_{n,2} \ast g$ where $g$ is given as after Theorem \ref{th:counterex}. Then $\excontW_n$ is a \textit{regular potential} and $\excontW_n \leq \excontV_{n,1}$. Theorem \ref{th:counterex} applied to $D_n:=4\sqrt{\lambda_2}/(n-1)$ shows for $R=n$, $\exP$-a.s.,
\begin{align*}
\sqrt n \alpha_{\excontW_n}(e_1)
& \leq \sqrt n \limsup_{u \to \infty} \frac{1}{u} a(u,\excontV_{n,1},\omega)\\
& = \sqrt n \limsup_{u \to \infty} \frac{1}{u} a(u,V_{1/n,n,1/n},\Phi_{1/n^2}(\omega))
\leq \sqrt 2 + \sqrt n D_n < \alpha_2(e_1)
\end{align*}
for $n$ large enough. This, continuity of the convolution, and \eqref{eq:counterex3} show that there is no typical scaling.

\vspace{0.5 cm}
\noindent
\thanks{\textbf{Acknowledgements:}
This work was funded by the ERC Starting Grant 208417-NCIRW.\\
Many thanks go to my supervisor Prof.\ Dr.\ Martin P.\ W.\ Zerner for his support. He suggested to me the core idea for the examples given in Section \ref{sec:counterex}. I am glad having had the opportunity to talk to Prof.\ Dr.\ Andrey L.\ Piatnitski about the topics of Section~\ref{sec:continuity}.

\small
\bibliographystyle{amsalpha}

\def\cprime{$'$} \def\cprime{$'$}
\providecommand{\bysame}{\leavevmode\hbox to3em{\hrulefill}\thinspace}
\providecommand{\MR}{\relax\ifhmode\unskip\space\fi MR }
\providecommand{\MRhref}[2]{%
  \href{http://www.ams.org/mathscinet-getitem?mr=#1}{#2}
}
\providecommand{\href}[2]{#2}

\end{document}